\documentclass[12pt]{article}
\usepackage{hyperref}
\usepackage{latexsym,amsmath}
\usepackage{graphicx}
\usepackage{caption}
\usepackage{float}
\usepackage{amsfonts}
\usepackage{isomath}
\usepackage{amssymb,array}
\usepackage{epsfig}
\usepackage{color}
\usepackage{setspace}
\usepackage{amsthm}
\usepackage{mathrsfs}
\usepackage[margin=1in]{geometry}
\usepackage{pdflscape}
\newtheorem{thm}{Theorem}[section]
\newtheorem{lem}{Lemma}[section]
\newtheorem{re}{Remark}[section]

\newtheorem{p}{Proposition}[section]

\newcommand{\keywords}[1]{\textbf{\textit{Keywords--}}#1}
\newcommand{\MSC}[1]{\textbf{\textit{Mathematics Subject Classification--}}#1}

\begin{document}
	
	\title{On estimation of weighted cumulative residual Tsallis entropy for complete and censored samples}
	\author{Siddhartha Chakraborty\thanks{\href{mailto:siddharthatdkr@gmail.com}{siddharthatdkr@gmail.com}} 
    \and Asok K. Nanda \thanks{\href{mailto:asok.k.nanda@gmail.com}{asok.k.nanda@gmail.com}}\\ Department of Mathematics and Statistics\\ Indian Institute of Science Education and Research Kolkata\\ West Bengal, India}
	\date{}
	\maketitle
	\begin{abstract}
		Recently, weighted cumulative residual Tsallis entropy has been introduced in the literature as a generalization of weighted cumulative residual entropy. We study some new properties of weighted cumulative residual Tsallis entropy measure. Next, we propose some non-parametric estimators of this measure. Asymptotic properties of these estimators are discussed. Performance of these estimators are compared by mean squared error. Non-parametric estimators for weighted cumulative residual entropy measure are also discussed. Estimator for weighted cumulative residual Tsallis entropy for progressive type-II censored data is proposed and its performance is investigated by Monte-Carlo simulations for various censoring schemes. Two uniformity tests for complete samples are proposed based on an estimator of these two measures and power of the tests are compared with some popular tests. The tests perform reasonably well. Uniformity test under progressively type-II censored data is also developed. Some real datasets are analyzed for illustration.
		
	\end{abstract}
	\\
	
	\noindent\MSC{62B10, 62G05, 94A17}
	
\section{Introduction}\label{S1}
Entropy is by far regarded as the most important concept in information theory which was first mathematically formulated by Shannon (1948) independently of Wiener (1948).  For a discrete random variable (rv) $X$ that takes values \{$x_1,x_2,\cdots,x_n$\} with probabilities $\mathbf{P}$ = \{$p_1,p_2,\cdots,p_n$\}, Shannon entropy is defined as 
\begin{equation}\label{e1}
H(\mathbf{P})=-\sum_{i=1}^{n}p_i\log p_i.
\end{equation}
For a continuous rv $X$, it is defined as
\begin{equation}\label{e2}
H(X)=-\int_{A} f(x)\log f(x)dx=E\left[-\log f(X)\right] ,
\end{equation}
where log is the natural logarithm, $0\log0=0$ for computational convenience, $f$ is the probability density function (pdf) of $X$, $A$ is the support of $X$ and $E$ is the expectation operator. It is also known as the differential entropy. It is considered as a measure of uncertainty and represents the average information contained in the underlying rv. Shannon's work on entropy gave rise to a new branch of applied probability with useful applications in a variety of fields such as economics, thermodynamics, statistical mechanics, mathematical biology, signal processing, statistics and reliability. An enormous amount of research has been carried out over the years by many scholars on generalizations of Shannon entropy and their applications. One of the important generalizations is due to Tsallis (1988), which is known as generalized Tsallis entropy and is given by 
\begin{equation}\label{e3}
T_{\alpha}(X)=\frac{1}{\alpha-1}\left( 1-\int_{0}^{+\infty}f^{\alpha}(x)dx\right),\;0<\alpha\neq1. 
\end{equation}
The parameter $\alpha$ is called the generalization parameter and when $\alpha\to 1$, $T_{\alpha}(X)\to H(X)$.\\ 

Entropy defined in Eq. (\ref{e1}) is always positive but differential entropy in Eq. (\ref{e2}) may be negative. Differential entropy has some drawbacks such as it cannot be defined for distributions that do not have densities. For discrete rvs $X$ and $Y$, conditional entropy of $X$ given $Y$ is zero if and only if (iff) $X$ is a function of $Y$. However, differential conditional entropy of $X$ given $Y$ is zero does not necessarily imply that $X$ is a function of $Y$. Rao et al. (2004) extended differential entropy by replacing the densities with the survival function of the rv. This measure is called cumulative residual entropy (CRE) and for a non-negative continuous rv $X$, it is defined as
\begin{equation}\label{e4}
\xi(X)=-\int_{0}^{+\infty}\bar{F}(x)\log \bar{F}(x)dx,
\end{equation}   
\noindent where $\bar{F}$ is the survival function of $X$. The CRE overcomes the above mentioned challenges that differential entropy faces and it possesses fundamental properties that Shannon entropy has. Like Shannon entropy, CRE is always non-negative and it increases by adding independent components (observations) and decreases by conditioning. Apart from that, CRE also possesses some useful properties such as it can be defined for both discrete and continuous rvs and can be easily estimated from data using empirical distribution function. Also CRE$(X|Y)=0$ iff $X$ is a function of $Y$ where CRE$(X|Y)$ is the conditional CRE of $X$ given $Y$.\\

 Ever since its introduction, various generalizations of CRE, analogous to the generalizations of Shannon entropy, have been considered by many authors. Sati and Gupta (2015) proposed one such generalization of the Tsallis entropy which is known as cumulative residual Tsallis entropy (CRTE) of order $\alpha$ and is defined as 
 \begin{equation}\label{e5}
 T_{\alpha}(X)=\frac{1}{\alpha-1}\left(1-\int_{0}^{+\infty}\bar{F}^{\alpha}(x) \right)dx,\;0<\alpha\neq 1. 
 \end{equation}
Rajesh and Sunoj (2019) proposed an alternative version of CRTE as 
\begin{equation}\label{e6}
\xi_{\alpha}(X)=\frac{1}{\alpha-1}\int_{0}^{+\infty}\left(\bar{F}(x)-\bar{F}^{\alpha}(x) \right)dx,\;0<\alpha\neq 1. 
\end{equation}
Both $T_{\alpha}(X)$ and $\xi_{\alpha}(X)$ reduce to CRE as $\alpha \to 1$. When $\alpha =2$, $\xi_{\alpha}(X)$ becomes Gini's mean difference (GMD). This alternative measure is more flexible than $T_{\alpha}(X)$ because it has more relationships with other measures related to reliability and information theory, see Rajesh and Sunoj (2019) and Toomaj and Atabay (2022).\\

 Entropy, CRE and their related measures consider only the quantitative (i.e. probabilistic) information but in many applied fields it is often required to consider qualitative characteristics or the utility of the random events as well. For example, in a two person game, it is necessary to take into account the various random strategies (quantitative) of the players involved as well as the gain or loss (qualitative) corresponding to the chosen strategy. Noticing the importance of qualitative characteristics of the random events in various applied fields, Belis and Guiasu (1968) introduced the concept of weighted entropy by assigning non-negative weights to each event based on their utility. For the continuous case, weighted differential entropy can be defined as 
\begin{equation}\label{e7}
H^w(X)=-\int_{0}^{+\infty}xf(x)\log f(x)dx,
\end{equation}
where the factor $x$ is the linear weight function that gives more importance to the larger values of the rv $X$. Motivated from the work of weighted entropy measure, numerous weighted  measures are introduced in the literature. Misagh et al. (2011) introduced weighted cumulative residual entropy (WCRE) measure which is defined as 
\begin{equation}\label{e8}
\xi^w(X)=-\int_{0}^{+\infty}x\bar{F}(x)\log \bar{F}(x)dx.
\end{equation} 
Chakraborty and Pradhan (2023) proposed weighted cumulative residual Tsallis entropy (WCRTE) as 

\begin{equation}\label{e9}
\xi_{\alpha}^w(X)=\frac{1}{\alpha-1}\int_{0}^{+\infty}x\left(\bar{F}(x)-\bar{F}^{\alpha}(x) \right)dx,\;0<\alpha\neq 1. 
\end{equation}
As $\alpha\to1$, it reduces to WCRE. Various properties of WCRTE measure along with its dynamic version and its potential as a risk measure for heavy tailed risk analysis was discussed in Chakraborty and Pradhan (2023). WCRTE is further studied by many authors, see, for example, Kattumannil et al. (2022, 2024), Zuo and Yin (2025) and the references therein.\\

Estimations of entropy and related measures were considered in literature by many authors; see, for example, Vasicek (1976), Ebrahimi et al. (1994), Correa (1995), Noughabi (2010), Al-Labadi et al. (2025) and the references therein. Recently, research on weighted information measures gained quite a lot of attention see, for example, Balakrishnan et al. (2022), Chakraborty (2023), Hashempour et al. (2023), Khammar and Jahanshahi (2018) and the references therin. However, estimations of weighted measures have not been studied to the same extent. The main focus of this article is to study various non-parametric estimators of WCRTE measure. Chakraborty and Pradhan (2023) introduced one non-parametric estimator of WCRTE measure. Here we propose four estimators of WCRTE, study their asymptotic properties and compare their performances. Also, we briefly discuss non-parametric estimations of WCRE measure. We introduce a non-parametric estimator for WCRTE under progressive type-II censored data. We develop two uniformity tests based on an estimator of WCRTE and WCRE measures for complete sample. Also, we develop an uniformity test for progressive type-II censored data as well. The rest of the paper is organised as follows:\\

We study some important properties of WCRTE measure in Section \ref{S2}. Non-parametric estimators of WCRTE and WCRE are introduced and their asymptotic properties are investigated in Section \ref{S3}. A simulation study is performed to compare the performance of the estimators in Section \ref{S4}. A non-parametric estimator for WCRTE under progressively type-II censored data is discussed in Section \ref{S5}. Uniformity tests are discussed in Section \ref{S6}. Finally, some concluding remarks are made in Section \ref{S7}.

\section{Properties of WCRTE}\label{S2} In this section, we study some properties of WCRTE measure. The following theorem discusses the existence of WCRTE measure.

\begin{thm}\label{th1}
For a non-negative continuous rv $X$, $\xi_{\alpha}^w(X)$ exists if variance of $X$ exists when $\alpha>1$.	
\end{thm}
\begin{proof}
	From Eq. (\ref{e9}), we have $$\xi_{\alpha}^w(X)=\frac{1}{\alpha-1}\left[\frac{E(X^2)}{2}-\int_{0}^{+\infty}x\bar{F}^{\alpha}(x)dx \right].$$
	Now,
	\begin{eqnarray}
	\int_{0}^{+\infty}x\bar{F}^{\alpha}(x)dx&=&\int_{0}^{1}x\bar{F}^{\alpha}(x)dx+\int_{1}^{+\infty}x\bar{F}^{\alpha}(x)dx\nonumber\\
	&\leq&\int_{0}^{1}xdx+\int_{1}^{+\infty}x\left[\frac{E(X^p)}{x^p}\right]^{\alpha}dx\nonumber\\
	&=& \frac{1}{2}+\left[E(X^p) \right]^{\alpha}\int_{1}^{+\infty}x^{1-p\alpha}dx.\nonumber\\
	&<&\infty,\nonumber  
	\end{eqnarray}
	
if $p>2/\alpha$. Therefore, it is easy to see that, for $\alpha>1$, $\xi_{\alpha}^w(X)$ exists when $E(X^2)$ exists i.e. the variance exists.
Hence the proof.
\end{proof}
We calculate WCRTE for some popular distributions and present them in Table \ref{tab1}.

\noindent Next we provide a lower bound for WCRTE measure in terms of entropy.

\begin{thm}
	For a non-negative continuous rv $X$, the following inequality holds:
	$$\xi_{\alpha}^w(X)\geq\exp\left[H(X)+E(\log X)+\eta(\alpha)\right],$$ where $\eta(\alpha)=\int_{0}^{1}\log\frac{u-u^{\alpha}}{\alpha-1}du.$
\end{thm}
\begin{proof}
	Using log-sum inequality we can write
	\begin{eqnarray}\label{e10}
	\int_{0}^{+\infty}f(x)\log\frac{f(x)}{\frac{1}{\alpha-1}x(\bar{F}(x)-\bar{F}^{\alpha}(x)}dx\geq\int_{0}^{+\infty}f(x)dx\;\log\frac{\int_{0}^{+\infty}f(x)dx}{\xi_{\alpha}^w(X)}=-\log\xi_{\alpha}^w(X).
	\end{eqnarray}
Now,
\begin{eqnarray}\label{e11}
\int_{0}^{+\infty}f(x)\log\frac{f(x)}{\frac{1}{\alpha-1}x(\bar{F}(x)-\bar{F}^{\alpha}(x)}dx &=& -H(X)-\int_{0}^{+\infty}f(x)\log\left[ \frac{1}{\alpha-1}x(\bar{F}(x)-\bar{F}^{\alpha}(x)\right]dx\nonumber\\
&=& -H(X)-\int_{0}^{+\infty}f(x)\log x \;dx\nonumber\\
&&\;\;\;\;-\int_{0}^{+\infty}f(x)\log\frac{\bar{F}(x)-\bar{F}^{\alpha}(x)}{\alpha-1}dx\nonumber\\
&=&-H(X)-E(\log X)-\int_{0}^{1}\log\frac{u-u^{\alpha}}{\alpha-1}du.
\end{eqnarray}
\noindent Combining Eq. (\ref{e10}) and (\ref{e11}) and after some simplifications, we get the result.
\end{proof}

\begin{table}[h!]
	\centering
	\caption{WCRTE for some distributions.}\label{tab1}
	\begin{tabular}{p{3cm} p{5.5cm} p{2cm} p{5.5cm}}
		\hline
		Distributions &$F(x)$ & Notation & $\xi_{\alpha}^w(X)$\\[1ex]
		\hline
		Uniform & $\frac{x}{\theta};\; 0<x<\theta$ & U(0,$\theta$) & $\frac{\theta^2(\alpha+4)}{6(\alpha+1)(\alpha+2)}$\\[1ex]
		Exponential & $1-e^{-\lambda x};\; x>0,\; \lambda >0 $ & Exp($\lambda$) & $\frac{\alpha+1}{\alpha\lambda ^2}$\\[1ex]
		Rayleigh & $1-e^{-\frac{x^2}{2\sigma^2}};\; x \geq 0,\; \sigma >0 $ & RA($\sigma$) & $\frac{\sigma^2}{\alpha}$ \\[2ex]
		Pareto I & $1-\left(\frac{k}{x}\right)^{\delta};\;x>k,\;\delta>0$ & PA($\delta,k$) & $\frac{\delta k^2}{(\delta-2)(\delta\alpha-2)}$, $\delta>(<)2,\alpha>(<)1$\\[2ex] 
		Weibull & $1-e^{-(\lambda x)^{p}},\;x>0,\;\lambda,\;p>0$ & WE($p,\lambda$) & $\frac{\Gamma\left( \frac{2}{p}\right)\left(1-\frac{1}{\alpha^{2/p}}\right)}{p\lambda^2(\alpha-1)}$\\ [2ex]
		\hline
	\end{tabular}
\end{table}
 Now we provide an alternative expression for WCRTE which will be later used to develop a non-parametric estimator.

\begin{thm}\label{th4}
	For a non-negative continuous rv $X$, the following relation holds:
	$$\xi_{\alpha}^w(X)=\frac{1}{2(\alpha-1)}\int_{0}^{+\infty}x^2\left(1-\alpha\bar{F}^{\alpha-1}(x)\right) dF(x).$$
\end{thm}
\begin{proof}
\begin{eqnarray}
\frac{1}{2(\alpha-1)}\int_{0}^{+\infty}x^2\left(1-\alpha\bar{F}^{\alpha-1}(x)\right) dF(x)
&=&\frac{1}{\alpha-1}\int_{0}^{+\infty}\left(\int_{0}^{x}vdv\right) \left(1-\alpha\bar{F}^{\alpha-1}(x)\right)dF(x)\nonumber\\
&=&\frac{1}{\alpha-1}\int_{0}^{+\infty}\left[\int_{v}^{+\infty}\left(1-\alpha\bar{F}^{\alpha-1}(x)\right)dF(x)\right]vdv\nonumber\\
%&=&\frac{1}{\alpha-1}\int_{0}^{+\infty}\left[F(x)+\bar{F}^{\alpha}(x)\right]_{v}^{+\infty} vdv\nonumber\\
&=&\frac{1}{\alpha-1}\int_{0}^{+\infty}v(\bar{F}(v)-\bar{F}^{\alpha}(v))dv.\nonumber
\end{eqnarray}	
Hence the result follows.
\end{proof}

\section{Non-parametric Estimation}\label{S3} In this section, we develop some non-parametric estimators for WCRTE and WCRE measures and study their asymptotic properties. Let $X_1,X_2,...,X_n$ be a random sample of size $n$ drawn from a distribution with cdf $F$ and let $X_{(1)}\leq X_{(2)}\leq...\leq X_{(n)}$ be the corresponding order statistics. Further let $F_n$ be the empirical distribution function of $X$. Then 
$$F_n(x)=\frac{\text{number of obs.} \leq x}{n}=
\begin{cases}
0,\;\;\;\;\mbox{if}\;x<X_{(1)},\\
\frac{i}{n},\;\;\;\mbox{if}\;X_{(i)}\leq x<X_{(i+1)},\;i=1,2,\cdots n-1,\\
1,\;\;\;\mbox{if}\;x \geq X_{(n)}.
\end{cases}$$
%Another definition of $F_n(x)$ can be given by $F_n(x)=\frac{\text{number of obs.} \leq x}{n+1}$. This is due to the fact that as $n\to +\infty$, $\frac{n}{n+1}\to1$.

\subsection{Estimation of WCRTE}

Using empirical distribution function, Chakraborty and Pradhan (2023) proposed an estimator of WCRTE as 

\begin{eqnarray}\label{e12}
\hat{\xi}_{\alpha}^w(X)&=&\frac{1}{(\alpha-1)}\int_{0}^{+\infty}x\left(\bar{F}_n(x)-\bar{F}^{\alpha}_n(x) \right)dx\nonumber\\ 
&=&\frac{1}{2(\alpha-1)}\sum_{i=1}^{n-1}\left( X^2_{(i+1)}-X^2_{(i)}\right) \left[ \left( 1-\frac{i}{n}\right) -\left( 1-\frac{i}{n}\right)^{\alpha}\right].
\end{eqnarray}

Now we propose two new estimators for WCRTE measure. These estimators are motivated from the works of Vasicek (1976) and Ebrahimi et al. (1994). Note that, using the integral transformation $F(x)=u$, we can express WCRTE in Eq. (\ref{e9}) as 

\begin{eqnarray}\label{e13}
\xi_{\alpha}^w(X)&=&\frac{1}{\alpha-1}\int_{0}^{1}F^{-1}(u)[1-u-(1-u)^{\alpha}]\;\frac{d}{du}\left(F^{-1}(u)\right)du\nonumber\\
&=&\frac{1}{2(\alpha-1)}\int_{0}^{1}[1-u-(1-u)^{\alpha}]\;\frac{d}{du}\left(F^{-1}(u)\right)^2du. 
\end{eqnarray}

Vasicek (1976) estimated $\frac{d}{du}\left(F^{-1}(u) \right)$ from the slope of the line joining the points $(F(X_{(i+m)}),X_{(i+m)})$ and $(F(X_{(i-m)}),X_{(i-m)})$. The window size $m$ is an integer and $m<\frac{n}{2}$. So we have

\begin{eqnarray}
	\frac{d}{du}\left(F^{-1}(u) \right)&\approx&\frac{X_{(i+m)}-X_{(i-m)}}{F(X_{(i+m)})-F(X_{(i-m)})}\nonumber\\
	&=&\frac{X_{(i+m)}-X_{(i-m)}}{\frac{i+m}{n}-\frac{i-m}{n}}\nonumber\\
	&=&\frac{n}{2m}(X_{(i+m)}-X_{(i-m)}).\nonumber
\end{eqnarray}
Note that, $X_{(i)}=X_{(1)}$ if $i<1$ and $X_{(i)}=X_{(n)}$ if $i>n$. 
Now we can estimate $\frac{d}{du}\left(F^{-1}(u)\right)^2$ as 

\begin{eqnarray}
\frac{d}{du}\left(F^{-1}(u)\right)^2&=& 2F^{-1}(u)\frac{d}{du}\left(F^{-1}(u)\right)\nonumber\\
&\approx& \frac{n}{2m}(X_{(i+m)}+X_{(i-m)})(X_{(i+m)}-X_{(i-m)})\nonumber\\
&=& \frac{n}{2m}\left( X^2_{(i+m)}-X^2_{(i-m)}\right) .\nonumber
\end{eqnarray}
Therefore, a Vasicek-type estimator for WCRTE is defined as
 
\begin{eqnarray}\label{e14}
\xi_{\alpha}^wV&=&\frac{1}{2(\alpha-1)}\frac{1}{n}\sum_{i=1}^{n}\left[1-\frac{i}{n}-\left( 1-\frac{i}{n}\right)^{\alpha}\right]\frac{n}{2m}\left( X^2_{(i+m)}-X^2_{(i-m)}\right) \nonumber\\
&=&\frac{1}{4m(\alpha-1)}\sum_{i=1}^{n}\left( X^2_{(i+m)}-X^2_{(i-m)}\right) \left[1-\frac{i}{n}-\left( 1-\frac{i}{n}\right)^{\alpha}\right].
\end{eqnarray}
Here $V$ in $\xi^w_{\alpha}(V)$ stands to emphasize that this estimator is a Vasicek-type estimator.\\

\noindent Ebrahimi et al. (1994) modified Vasicek entropy estimator with the help of a weight function in order to estimate the slope more accurately at the extreme points. The weight function of Ebrahimi is

$$C_i=
\begin{cases}
1+\frac{i-1}{m},\;\;\;\;\mbox{if}\;\;\;1\leq i\leq m,\\
2,\;\;\;\;\;\;\;\;\;\;\;\;\;\mbox{if}\;\;\;m+1\leq i \leq n-m,\\
1+\frac{n-i}{m},\;\;\;\;\mbox{if}\;\;\;n-m+1\leq i \leq n.
\end{cases}$$

Using these weights, another estimator of WCRTE can be defined as 

\begin{eqnarray}\label{e15}
\xi_{\alpha}^wE&=&\frac{1}{2(\alpha-1)}\frac{1}{n}\sum_{i=1}^{n}\left[1-\frac{i}{n}-\left( 1-\frac{i}{n}\right)^{\alpha}\right]\frac{n}{C_im}\left( X^2_{(i+m)}-X^2_{(i-m)}\right) \nonumber\\
&=&\frac{1}{2m(\alpha-1)}\sum_{i=1}^{n}\frac{X^2_{(i+m)}-X^2_{(i-m)}}{C_i} \left[1-\frac{i}{n}-\left( 1-\frac{i}{n}\right)^{\alpha}\right].
\end{eqnarray}
Here $E$ stands for Ebrahimi as $V$ stands for Vasicek in  (\ref{e14}).\\

\noindent We introduce another estimator of WCRTE measure by modifying Ebrahimi weights in $\xi_{\alpha}^wE$. The modifications happen in the estimation of $\frac{d}{du}\left(F^{-1}(u)\right)^2$ as follows:

\begin{eqnarray}
\frac{d}{du}\left(F^{-1}(u)\right)^2&=& 2F^{-1}(u)\frac{d}{du}\left(F^{-1}(u)\right)\nonumber\\
&\approx&2 \frac{(X_{(i+m)}+X_{(i-m)})}{C_i} \frac{n}{C_im}(X_{(i+m)}-X_{(i-m)})\nonumber\\
&=&\frac{2n}{m}\left( \frac{X^2_{(i+m)}-X^2_{(i-m)}}{C^2_i}\right).\nonumber
\end{eqnarray}
Here the modifications are made in the estimations of $\frac{d}{du}\left(F^{-1}(u)\right)^2$ as well as $F^{-1}(u)$. The new estimator is given by

\begin{eqnarray}\label{e15A}
\xi_{\alpha}^wN&=&\frac{1}{2(\alpha-1)}\frac{1}{n}\sum_{i=1}^{n}\left[1-\frac{i}{n}-\left( 1-\frac{i}{n}\right)^{\alpha}\right]\frac{2n}{m}\left( \frac{X^2_{(i+m)}-X^2_{(i-m)}}{C^2_i}\right) \nonumber\\
&=&\frac{1}{m(\alpha-1)}\sum_{i=1}^{n}\frac{X^2_{(i+m)}-X^2_{(i-m)}}{C^2_i} \left[1-\frac{i}{n}-\left( 1-\frac{i}{n}\right)^{\alpha}\right].
\end{eqnarray}

The following theorem shows the consistency of these estimators. Proof is ovbious from the works of Vasicek (1976), and hence omitted.

\begin{thm}
  Let $X_1,X_2,\cdots,X_n$ be a random sample from a continuous distribution having cdf $F$, pdf $f$ and finite second moment. Then, for $\alpha>1$,
  
  \begin{enumerate}
  	\item $\xi_{\alpha}^wV\stackrel{p}{\to}\xi_{\alpha}^w(X)$, 
  	\item $\xi_{\alpha}^wE\stackrel{p}{\to}\xi_{\alpha}^w(X)$,
  	\item $\xi_{\alpha}^wN\stackrel{p}{\to}\xi_{\alpha}^w(X)$,
  \end{enumerate}
  as $m,n\to+\infty$, such that $\frac{m}{n}\to0$. The notation ``$\stackrel{p}{\to}$" indicates the convergence in probability.
\end{thm}

Next, we define an estimator for WCRTE, by making use of Theorem \ref{th4}, as

\begin{eqnarray}\label{e16}
\xi_{\alpha}^wL&=&\frac{1}{2(\alpha-1)}\int_{0}^{+\infty}x^2\left(1-\alpha\bar{F}_n^{\alpha-1}(x) \right)dF_n(x)\nonumber\\
&=& \frac{1}{2(\alpha-1)}\frac{1}{n}\sum_{i=1}^{n}X^2_{(i)}\left[1-\alpha\left(1-\frac{i}{n} \right)^{\alpha-1} \right].
\end{eqnarray}
This estimator belongs to the class of estimators called linear combination of functions of order statistics.

The following proposition discusses the effect of scale transformation on the estimators for WCRTE measure.
\begin{p}\label{p1}
	Let $X_1,X_2,\cdots,X_n$ be a random sample with WCRTE $\xi_{\alpha}^w(X)$ and let $Y_i=\theta X_i,$ $\theta>0,\;i=1(1)n$. Suppose the estimators for WCRTE of $X$ and $Y$ are $\xi_{\alpha}^wV(X),\;\xi_{\alpha}^wE(X),\;\xi_{\alpha}^wN(X),\;\xi_{\alpha}^wL(X)$ and $\xi_{\alpha}^wV(Y),\;\xi_{\alpha}^wE(Y),\;\xi_{\alpha}^wN(Y),\;\xi_{\alpha}^wL(Y)$, respectively. Then 
	
	\begin{enumerate}
		\item $\xi_{\alpha}^wV(Y)=\theta^2\xi_{\alpha}^wV(X)$;
		\item $\xi_{\alpha}^wE(Y)=\theta^2\xi_{\alpha}^wE(X)$;
		\item $\xi_{\alpha}^wN(Y)=\theta^2\xi_{\alpha}^wN(X)$;
		\item $\xi_{\alpha}^wL(Y)=\theta^2\xi_{\alpha}^wL(X)$.
	\end{enumerate}
\end{p}
\begin{proof}
	From Eq. (\ref{e14}), we have
	\begin{eqnarray}
	\xi_{\alpha}^wV(Y)&=&\frac{1}{4m(\alpha-1)}\sum_{i=1}^{n}\left(Y^2_{(i+m)}-Y^2_{(i-m)}\right) \left[1-\frac{i}{n}-\left( 1-\frac{i}{n}\right)^{\alpha}\right]\nonumber\\
	&=&\frac{1}{4m(\alpha-1)}\sum_{i=1}^{n}\left(\theta^2X_{(i+m)}-\theta^2X_{(i-m)}\right) \left[1-\frac{i}{n}-\left( 1-\frac{i}{n}\right)^{\alpha}\right]\nonumber\\
	&=& \theta^2\xi_{\alpha}^wV(X).\nonumber
	\end{eqnarray}
	Similarly, we can prove the other results.
\end{proof}
\subsubsection{Asymptotic normality of $\xi^w_{\alpha}L$}
Here we study the asymptotic normality of $\xi_{\alpha}L$. Consider
\begin{eqnarray}\label{e17}
T_n=2(\alpha-1)\xi_{\alpha}^wL&=&\frac{1}{n}\sum_{i=1}^{n}X^2_{(i)}\left[1-\alpha\left(1-\frac{i}{n+1} \right)^{\alpha-1}\right]\nonumber\\
&=&\frac{1}{n}\sum_{i=1}^{n}\psi\left(\frac{i}{n+1}\right)h(X_{(i)}), 
\end{eqnarray}
where $\psi(u)=1-\alpha(1-u)^{\alpha-1}$ and $h(u)=u^2$. The following theorem discusses the asymptotic normality of $T_n$ and hence of $\xi_{\alpha}^wL$. The notation ``$\stackrel{d}{\to}$" indicates convergence in distribution.

\begin{thm}\label{th3.2}
	For a non-negative, absolutely continuous rv $X$, $\sqrt{n}\left(\xi_{\alpha}^wL-\xi_{\alpha}^w(X) \right)$ is asymptotically normally distributed with zero mean and variance $$\sigma^2=\frac{2}{(\alpha-1)^2}\int_{0}^{+\infty}\int_{0}^{+\infty}xyF(x)\bar{F}(y)\left[1-\alpha\bar{F}^{\alpha-1}(x) \right] \left[1-\alpha\bar{F}^{\alpha-1}(y) \right]dxdy.$$
\end{thm}
\begin{proof}
	From the results given in Chernoff et al. (1967) and Shorack (1969), it can be shown that $T_n$ converges to normal distribution since $\sqrt{n}(T_n-\mu)\stackrel{d}{\to} N(0,\sigma_1^2)$ as $n\to+\infty$. Note that, $E(T_n)=\mu$ and it can be obtained as
	\begin{eqnarray}\label{e18}
	\mu&=&\int_{0}^{1}\psi(u)h(F^{-1}(u))du\nonumber\\
	&=&\int_{0}^{1}\left(1-\alpha(1-u)^{\alpha-1}\right) (F^{-1}(u))^2du\nonumber\\
	&=&2(\alpha-1)\xi_{\alpha}^w(X).
	\end{eqnarray}
The variance of $T_n$ is 

\begin{eqnarray}\label{e19}
\sigma_1^2&=&2\int_{0}^{1}\int_{u}^{1}\psi(u)\psi(w)\left(h(F^{-1}(u))\right)'\left(h(F^{-1}(w))\right)'u(1-w)dwdu\nonumber\\
&=&8\int_{0}^{1}\int_{u}^{1}\psi(u)\psi(w)u(1-w)F^{-1}(u)F^{-1}(w)dF^{-1}(w)dF^{-1}(u)\nonumber\\
&=&8\int_{0}^{+\infty}\int_{y}^{+\infty}xyF(y)\bar{F}(x)\left[1-\alpha\bar{F}^{\alpha-1}(x) \right] \left[1-\alpha\bar{F}^{\alpha-1}(y) \right]dxdy.
\end{eqnarray}	
The result follows from equations (\ref{e18}) and (\ref{e19}) and by using the property of normal distribution.
\end{proof}
Since $\sigma^2$ contains unknown parameters in $F$, one needs to work with its estimator for pracical considerations. A consistent estimator for $\sigma^2$ based on a dataset can be defined as 
\begin{eqnarray}
\hat{\sigma}^2&=&\frac{2}{(\alpha-1)^2}\sum_{j=1}^{n-1}\sum_{i=j+1}^{n-1}\frac{j}{n}\left( 1-\frac{i}{n}\right) \left[1-\alpha\left( 1-\frac{i}{n}\right)^{\alpha-1} \right] \left[1-\alpha\left( 1-\frac{j}{n}\right)^{\alpha-1} \right]\nonumber\\
&&\;\;\;\;\;\;\;\;\;\;\;\;\;\;\;\;\;\;\;\;\;\;\;\;\;\;\;\;\;\;\;\;\;\;*\left( X^2_{(i+1)}-X^2_{(i)}\right)\left( X^2_{(j+1)}-X^2_{(j)}\right).\nonumber 
\end{eqnarray}

\subsection{Estimation of WCRE} WCRE is a limiting case of WCRTE measure. Like we expresssed WCRTE in Eq. (\ref{e13}) by using integral transformation $F(x)=u$, we can similarly expressed WCRE as

\begin{eqnarray}\label{e20}
\xi^w(X)=-\frac{1}{2}\int_{0}^{1}(1-u)\log(1-u)\;\frac{d}{du}\left(F^{-1}(u)\right)^2du.
\end{eqnarray}

From Eq. (\ref{e20}), we can develop the estimators for WCRE. The Vasicek-type and the Ebrahimi-type estimators of WCRE can be obtained as

\begin{eqnarray}
\xi^wV=-\frac{1}{4m}\sum_{i=1}^{n-1}\left( X^2_{(i+m)}-X^2_{(i-m)}\right) \left(1-\frac{i}{n}\right)\log\left( 1-\frac{i}{n}\right);\nonumber
\end{eqnarray}
and
\begin{eqnarray}
\xi^wE=-\frac{1}{2m}\sum_{i=1}^{n-1}\frac{ X^2_{(i+m)}-X^2_{(i-m)}}{C_i} \left(1-\frac{i}{n}\right)\log\left( 1-\frac{i}{n}\right),\nonumber
\end{eqnarray}
respectively. We define a new estimator for WCRE as
\begin{eqnarray}
\xi^wN=-\frac{1}{m}\sum_{i=1}^{n-1}\frac{ X^2_{(i+m)}-X^2_{(i-m)}}{C^2_i} \left(1-\frac{i}{n}\right)\log\left(1-\frac{i}{n}\right).\nonumber
\end{eqnarray}
\begin{re}
 As $\alpha\to1$, $\xi_{\alpha}^w(X)\to\xi^w(X)$. This limiting condition also holds for the respective estimators as well. To be specific, $\xi_{\alpha}^wV\to\xi^wV$, $\xi_{\alpha}^wE\to\xi^wE$ and $\xi_{\alpha}^wN\to\xi^wN$ as $\alpha\to1$.
\end{re}

\begin{thm}\label{th3.3}
	For a non-negative continuous rv $X$,
	\begin{eqnarray}
	\xi^w(X)=-\frac{1}{2}\int_{0}^{+\infty}x^2\left(1+\log\bar{F}(x)\right)dF(x.\nonumber) 
	\end{eqnarray}
\end{thm}
\begin{proof}
     Proof follows by proceeding along the same line as in Theorem \ref{th4}.
\end{proof}
Using Theorem \ref{th3.3}, we define an estimator for WCRE as
\begin{eqnarray}
\xi^wL%&=&-\frac{1}{2}\int_{0}^{+\infty}x^2\left(1+\log\bar{F}_n(x)\right)dF_n(x)\nonumber\\
&=&-\frac{1}{2n}\sum_{i=1}^{n}X^2_{(i)}\left[1+\log\left(1-\frac{i}{n}\right)\right]. 
\end{eqnarray}

In the following proposition, we discuss the effect of scale transformation on the estimators of WCRE measure. Proof is similar to that of Proposition \ref{p1}, and hence omitted.

\begin{p}\label{p2}
	Let $X_1,X_2,\cdots,X_n$ be a random sample with WCRE $\xi^w(X)$ and let $Y_i=\theta X_i,\;\theta>0,\;i=1(1)n$. Suppose the estimators for WCRE of $X$ and $Y$ are $\xi^wV(X)$, $\xi^wE(X)$, $\xi^wN(X)$, $\xi^wL(X)$ and $\xi^wV(Y),\;\xi^wE(Y),\;\xi^wN(Y),\;\xi^wL(Y)$, respectively. Then 
	
	\begin{enumerate}
		\item[(i)] $\xi^wV(Y)=\theta^2\xi^wV(X)$;
		\item[(ii)] $\xi^wE(Y)=\theta^2\xi^wE(X)$;
		\item[(iii)] $\xi^wN(Y)=\theta^2\xi^wN(X)$;
		\item[(iv)] $\xi^wL(Y)=\theta^2\xi^wL(X)$.
	\end{enumerate}
\end{p}

In the next theorem, we discuss the asymptotic normality of the estimator $\xi^wL$.
\begin{thm}
	For a non-negative continuous rv $X$, $\sqrt n \left(\xi^wL-\xi^w(X) \right)$ converges to $N(0,\sigma^2)$ as $n\to+\infty$ where $$\sigma^2=2\int_{0}^{+\infty}\int_{y}^{+\infty}xyF(y)\bar{F}(x)\left[1+\log\bar{F}(x)\right] \left[1+\log\bar{F}(y) \right]dxdy.$$
\end{thm}
\begin{proof}
	Consider, $\zeta_n=-2\xi^wL$. Now proceeding alomg the same lines as in Theorem \ref{th3.2}, it can be shown that $\sqrt n(\zeta_n-\mu)\to N(0,\sigma_1^2)$ as $n\to+\infty$, where the mean is $\mu=-2\xi^w(X)$ and the variance is $$\sigma_1^2=8\int_{0}^{+\infty}\int_{y}^{+\infty}xyF(y)\bar{F}(x)\left[1+\log\bar{F}(x)\right] \left[1+\log\bar{F}(y) \right]dxdy.$$
	Hence the proof follows.
\end{proof}
A consistent estimator for $\sigma^2$ can be obtained as
\begin{eqnarray}
\hat{\sigma^2}&=&2\sum_{j=1}^{n-1}\sum_{i=j+1}^{n-1}\frac{j}{n}\left( 1-\frac{i}{n}\right) \left( 1+\log\left( 1-\frac{i}{n}\right) \right)\left( 1+\log\left( 1-\frac{j}{n}\right) \right)\nonumber\\
&&\;\;\;\;\;\;\;\;\;\;\;\;\;\;\;\;\;\;\;\;\;\;\;\;\;\;\;\;\;\;\;\;\;\;*\left( X^2_{(i+1)}-X^2_{(i)}\right)\left( X^2_{(j+1)}-X^2_{(j)}\right).\nonumber 
\end{eqnarray}

\section{Comparison of estimators}\label{S4} We conduct a simulation study to assess the performance of the proposed estimators. We compare the performance of these estimators by means of Bias and mean square error (MSE). For reference distributions, we consider exponential, Weibull and uniform distributions. We denote them by Exp($\lambda$), WE($p,\lambda$) and U($a,b$), respectively. The cdfs of these distributions are provided in Table \ref{tab1}. We generate 10000 random samples of size $n$ = 10, 20 and 30 from Exp(1), Exp(2), U(0,1) and WE(2,1) distributions and calculate bias and MSEs of these estimators. The results are provided in Tables \ref{tab3}-\ref{tab6}.\\

The parameter of WCRTE is taken as $\alpha$ = 2. We choose $\alpha$ = 2 because from Theorem \ref{th1} we observe that, for a continuous distribution, WCRTE exists if variance of that distribution exists when $\alpha>$1. When $\alpha<$1, we will need higher order moments to exist for WCRTE to exist. In practice, it is better to work with $\alpha>$1, since higher the value of $\alpha$, it is more likely for WCRTE to exist (see Theorem \ref{th1}).\\ 

The estimators $\hat{\xi}^w_{\alpha}(X)$ and $\xi_{\alpha}^wL$ do not require a window size. However, $\xi_{\alpha}^wV$, $\xi_{\alpha}^wE$ and $\xi_{\alpha}^wN$ depend on the window size $m(<\frac{n}{2})$. We consider all possible values of $m$ for bias and MSE calculation and try to provide a guideline for the choice of $m$. The choice of $m$ will depend on the underlying distribution and the sample size $n$. We highlight the minimum values of MSEs of these estimators in Tables \ref{tab3}-\ref{tab6}.\\

From Tables \ref{tab3} and \ref{tab4}, it is observed that, for Exp(1) and Exp(2) distributions, $\xi_{\alpha}^wL$ performs the best for $n$ = 10 and $\xi_{\alpha}^wN$ performs the best for $n$ = 20 and 30. All these estimators perform much better for Exp(2) distribution than Exp(1). When the sample size increases the bias and the MSE of all the estimators decrease. The optimal choice of $m$ (for different $n$) for each estimator can be obtained from the respective tables. It is not possible to provide an exact formula for the choice of $m$ but approximate formulae may be provided from Tables \ref{tab3}-\ref{tab4}, which will provide optimal or near optimal choices of $m$ for exponential distributions. Note from the tables that
\begin{itemize}
	\item[(a)] for $\xi_{\alpha}^wV$ and $\xi_{\alpha}^wE$, when $n\leq$ 20, choose $m=[\frac{n}{2}]-1$. When $n$ = 30, choose $m=[\frac{n}{3}]$;
	\item[(b)] for $\xi_{\alpha}^wN$, choose $m=[\frac{n}{4}]+1$, for all $n$.
\end{itemize}

From Table \ref{tab5} it is observed that, for U(0,1) distribution $\xi_{\alpha}^wN$ performs the best, although the performance of all five estimators is similar. Performance of $\xi_{\alpha}^wV$, $\xi_{\alpha}^wE$ and $\xi_{\alpha}^wN$ do not significantly vary for the choices of $m$ and small values of $m$ will be appropriate. Also, for U(0,1), these estimators perform better than all the other models considered. As there is no significant differnces bettwen the MSEs of these five estimators for U(0,1) model, any choice among these estimator will work. For simplicity, its best to work with $\hat{\xi}^w_{\alpha}(X)$ or $\xi_{\alpha}^wL$.\\

From Table \ref{tab6}, we observe that for WE(2,1) (Rayleigh) distribution, $\xi_{\alpha}^wN$ may be considered to be the best whereas $\xi_{\alpha}^wL$ is the worst. However the difference in performance may be considered to be insignificant. For $\xi_{\alpha}^wV$, $\xi_{\alpha}^wE$ and $\xi_{\alpha}^wN$, we prescribe to choose small $m$ for better result.

%From the simulation study we find that $\xi_{\alpha}^wN$ performs the best, $\xi_{\alpha}^wV$ performs the worst and performance of $\xi_{\alpha}^wE$ stays in the middle for all the distributions considered. Also they perform better than $\hat{\xi}_{\alpha}^w(X)$ and $\xi_{\alpha}^wL$ in most of the cases. When sample size increases MSEs of all estimators decrease which is obvious. In general, $\xi_{\alpha}^wN$ has the lowest MSE among all other estimators. However, $\xi_{\alpha}^wN$ is more sensitive to the choice of $m$ than $\xi_{\alpha}^wV$ and $\xi_{\alpha}^wE$, whereas $\xi_{\alpha}^wV$ is the least sensitive to the choice of $m$. So, in order to work with $\xi_{\alpha}^wN$ and $\xi_{\alpha}^wE$, appropriate choice of $m$ is important for better results. Also, if the sample size is large then all estimators perform similarly. So, it is better to use $\xi_{\alpha}^wL$ for large $n$ because its asymptotic distribution is normal. This will be very useful in various statistical inference problems.

\subsection{Data Analysis}\label{D1} We analyze a real dataset, given in Nelson (2005, page 105), to study the effectiveness of the proposed estimators. The following data describe failure times (in minutes) for an insulating fluid between two electrodes subject to a voltage of 34 kV. 

\begin{table}[h!]
	\centering
	\begin{tabular}{cccccccccc}
		0.19, & 0.78, & 0.96, & 1.31, & 2.78, & 3.16, & 4.15, & 4.67, & 4.85, & 6.50, \\
		7.35, & 8.01, & 8.27, & 12.06, & 31.75, & 32.52, & 33.91, & 36.71, & 72.89. &
	\end{tabular}
\end{table}
\noindent Here sample size is 19. The dataset has come from an exponential distribution with density $$f(x)=\lambda \exp(-\lambda x);\;x,\lambda>0.$$ The maximum likelihood estimate (MLE) is $\hat{\lambda}=\frac{1}{\bar{X}}=0.0696$. We apply Kolmogorov-Smirnov test to check exponentiality of this dataset. The obtained p-value is 0.1672 which supports the exponentiality hypothesis. The theoretical value of WCRTE for exponential distribution is $\frac{\alpha+1}{\alpha\lambda^2}$.\\

We consider $\alpha$ = 2 for illustration. Then, the parametric estimate of WCRTE is $\frac{\alpha+1}{\alpha\hat{\lambda}^2}$ = 309.6512. We calculate non-parametric estimates for this dataset. For $m$-spacings estimators, we choose that $m$ which minimizes the absolute difference between the non-parametric estimates and the parametric estimate. The obtained values of the estimates are $\xi_{\alpha}^w(X)$ = 218.2337, $\xi_{\alpha}^wL$ = 232.5466, $\xi_{\alpha}^wV$ = 294.3883 for $m$ = 9, $\xi_{\alpha}^wE$ = 309.6309 for $m$ = 8 and $\xi_{\alpha}^wN$ = 304.8546 for $m$ = 5. So the estimators $\xi_{\alpha}^wE$ and $\xi_{\alpha}^wL$ perform well. Note that $m$-spacings estimators perform much better than $\xi_{\alpha}^w(X)$ and $\xi_{\alpha}^wL$.\\

For this dataset, Vasicek-type estimator performs best for maximum value of $m$, Ebrahimi-type estimator performs best for second largest choice of $m$ and $\xi_{\alpha}^wN$ performs best for $m$ = 5, which is the middle point. These values of $m$ are similar to the values found in Table \ref{tab4} for Exp(1) distribution for $n$ = 20. They are in line with the choices provided in Section \ref{S4}.

\begin{landscape}
\begin{table}[h!]
	\centering
	\caption{Bias and MSE of the estimators for Exp(1) distribution with $\alpha$ = 2, $n$ = 10, 20 and 30 and for various $m$.}\label{tab3}
	\begin{tabular}{c c c c c c c c c c c c}
		\hline
		& \multicolumn{2}{c}{$\hat{\xi}_{\alpha}^w(X)$}&\multicolumn{2}{c}{$\xi_{\alpha}^wL$} & & \multicolumn{2}{c}{$\xi_{\alpha}^wV$}&\multicolumn{2}{c}{$\xi_{\alpha}^wE$} & \multicolumn{2}{c}{$\xi_{\alpha}^wN$}\\
		\hline
		$n$ & Bias & MSE & Bias & MSE & $m$ & Bias & MSE & Bias & MSE & Bias & MSE\\
		\hline
		10 & -0.8259 & \bf0.9686 & -0.7251 & \bf0.8848 & 1 & -1.005 & 1.1386 & -1.003 & 1.1301 & -1.001 & 1.1314\\
		   & & & & & 2 & -0.9764 & 1.1045 & -0.9301 & 1.0660 & -0.8945 & 1.0177\\
		   & & & & & 3 & -0.9564 & 1.1038 & -0.8787 & 1.0321 & -0.7865 & 0.9478\\
		   & & & & & 4 & -0.9365 & \bf 1.0956 & -0.8373 & \bf 0.9963 & -0.6966 & \bf 0.9469\\[1ex]
		   
		20 & -0.7879 & \bf0.7751 & -0.7400 & \bf0.7205 & 1 & -0.9313 & 0.9486 & -0.9315 & 0.9568 & -0.9281 & 0.9495\\
		& & & & & 2 & -0.8965 & 0.9092 & -0.8801 & 0.8866 & -0.8489 & 0.8500\\
		& & & & & 3 & -0.8653 & 0.8713 & -0.8234 & 0.8213 & -0.7680 & 0.7732\\
		& & & & & 4 & -0.8379 & 0.8515 & -0.7746 & 0.7817 & -0.6852 & 0.7200\\
		& & & & & 5 & -0.8146 & 0.8290 & -0.7223 & 0.7490 & -0.6122 & \bf 0.6816\\
		& & & & & 6 & -0.7804 & 0.8134 & -0.6824 & 0.7354 & -0.5342 & 0.7161\\
		& & & & & 7 & -0.7609 & 0.8166 & -0.6364 & 0.7442 & -0.4705 & 0.7278\\
		& & & && 8 & -0.7490 & 0.8090 & -0.5964 & 0.7315 & -0.4171 & 0.6940\\
		& & & & & 9 & -0.7405 & \bf 0.8078 & -0.5893 & \bf 0.7189 & -0.3539 & 0.7440\\[1ex]
		
		30 & -0.7749 & \bf0.7143 & -0.7427 & \bf0.6693 & 1 & -0.8943 & 0.8679 & -0.8910 & 0.8591 & -0.8951 & 0.8664\\
		& & & & & 2 & -0.8701 & 0.8334 & -0.8503 & 0.8087 & -0.8287 & 0.7801\\
		& & & & & 3 & -0.8420 & 0.7972 & -0.8062 & 0.7529 & -0.7595 & 0.7047\\
		& & & & & 4 & -0.8177 & 0.7709 & -0.7566 & 0.7080 & -0.6917 & 0.6440\\
		& & & & & 5 & -0.7860 & 0.7406 & -0.7174 & 0.6706 & -0.6287 & 0.6050\\
		& & & & & 6 & -0.7522 & 0.7716 & -0.6673 & 0.6435 & -0.5634 & 0.5837\\
		& & & & & 7 & -0.7283 & 0.7019 & -0.6328 & 0.6234 & -0.4957 & \bf 0.5722\\
		& & & & & 8 & -0.7120 & 0.6864 & -0.5855 & 0.6152 & -0.4329 & 0.5796\\
		& & & & & 9 & -0.6765 & 0.6700 & -0.5413 & 0.6180 & -0.3744 & 0.5818\\
		& & & & & 10& -0.6584 & 0.6712 & -0.5212 & 0.5994 & -0.3284 & 0.6004\\
		& & & & & 11& -0.6430 & 0.6522 & -0.4890 & \bf 0.5917 & -0.2638 & 0.6726\\
		& & & & & 12& -0.6272 & \bf 0.6498 & -0.4635 & 0.6047 & -0.2185 & 0.7171\\
		& & & & & 13& -0.6116 & 0.6780 & -0.4179 & 0.6285 & -0.1679 & 0.7245\\
		& & & & & 14& -0.5919 & 0.6682 & -0.4083 & 0.6152 & -0.1098 & 0.7953\\\hline

	\end{tabular}
\end{table}

\end{landscape}

\begin{landscape}
	\begin{table}[h!]
		\centering
		\caption{Bias and MSE of the estimators for Exp(2) distribution with $\alpha$ = 2, $n$ = 10, 20 and 30 and for various $m$.}\label{tab4}
		\begin{tabular}{c c c c c c c c c c c c}
			\hline
			& \multicolumn{2}{c}{$\hat{\xi}_{\alpha}^w(X)$}&\multicolumn{2}{c}{$\xi_{\alpha}^wL$} & & \multicolumn{2}{c}{$\xi_{\alpha}^wV$}&\multicolumn{2}{c}{$\xi_{\alpha}^wE$} & \multicolumn{2}{c}{$\xi_{\alpha}^wN$}\\
			\hline
			$n$ & Bias & MSE & Bias & MSE & $m$ & Bias & MSE & Bias & MSE & Bias & MSE\\
		\hline
		 10 & -0.2076 & \bf0.0616 & -0.1814 & \bf0.0560 & 1 & -0.2504 & 0.0713 & -0.2497 & 0.0706 & -0.2506 & 0.0707\\
		& & & & & 2 & -0.2444 & 0.0697 & -0.2366 & 0.0671 & -0.2206 & 0.0630\\
		& & & & &  3 & -0.2371 & 0.0684 & -0.2186 & 0.0646 & -0.1955 & \bf0.0595\\
		& & & & &  4 & -0.2347 & \bf0.0682 & -0.2097 & \bf0.0640 & -0.1734 & 0.0596\\[1ex]
		
		20 & -0.1964 & \bf0.0491 & -0.1836 & \bf0.0451 & 1 & -0.2327 & 0.0597 & -0.2323 & 0.0593 & -0.2312 & 0.0589\\
		& & & & &  2 & -0.2238 & 0.0566 & -0.2192 & 0.0550 & -0.2103 & 0.0526\\
		& & & & &  3 & -0.2172 & 0.0553 & -0.2054 & 0.0515 & -0.1917 & 0.0486\\
		& & & & &  4 & -0.2097 & 0.0531 & -0.1935 & 0.0493 & -0.1712 & 0.0450\\
		& & & & &  5 & -0.2026 & 0.0524 & -0.1796 & 0.0473 & -0.1529 & 0.0445\\
		& & & & &  6 & -0.1962 & 0.0513 & -0.1684 & 0.0466 & -0.1348 & \bf0.0442\\
		& & & & &  7 & -0.1912 & 0.0508 & -0.1606 & \bf0.0454 & -0.1158 & 0.0454\\
		& & & & &  8 & -0.1868 & \bf0.0507 & -0.1514 & 0.0458 & -0.0985 & 0.0447\\
		& & & & &  9 & -0.1837 & 0.0509 & -0.1450 & 0.0455 & -0.0864 & 0.0480\\[1ex]
		
		30 & -0.1940 & \bf0.0443 & -0.1847 & \bf0.0416 & 1 & -0.2234 & 0.0541 & -0.2242 & 0.0544 & -0.2247 & 0.0546\\
		& & & & &  2 & -0.2176 & 0.0520 & -0.2129 & 0.0504 & -0.2078 & 0.0489\\
		& & & & &  3 & -0.2084 & 0.0494 & -0.2038 & 0.0478 & -0.1904 & 0.0443\\
		& & & & &  4 & -0.2025 & 0.0474 & -0.1904 & 0.0443 & -0.1748 & 0.0408\\
		& & & & &  5 & -0.1960 & 0.0459 & -0.1794 & 0.0422 & -0.1549 & 0.0378\\
		& & & & &  6 & -0.1883 & 0.0441 & -0.1672 & 0.0400 & -0.1408 & 0.0368\\
		& & & & &  7 & -0.1830 & 0.0433 & -0.1573 & 0.0386 & -0.1237 & \bf0.0361\\
		& & & & &  8 & -0.1753 & 0.0421 & -0.1486 & 0.0377 & -0.1078 & \bf0.0361\\
		& & & & &  9 & -0.1716 & 0.0423 & -0.1381 & \bf0.0372 & -0.0944 & 0.0382\\
		& & & & &  10& -0.1642 & 0.0414 & -0.1293 & 0.0378 & -0.0776 & 0.0394\\
		& & & & &  11& -0.1610 & 0.0412 & -0.1209 & 0.0391 & -0.0671 & 0.0401\\
		& & & & &  12& -0.1576 & \bf0.0411 & -0.1156 & 0.0389 & -0.0565 & 0.0418\\
		& & & & &  13& -0.1537 & 0.0412 & -0.1022 & 0.0394 & -0.0404 & 0.0440\\
		& & & & &  14& -0.1493 & 0.0414 & -0.1008 & 0.0393 & -0.0306 & 0.0487\\\hline  
	\end{tabular}
\end{table}
\end{landscape}

\begin{landscape}
\begin{table}[h!]
	\centering
	\caption{Bias and MSE of the estimators for U(0,1) distribution with $\alpha$ = 2, $n$ = 10, 20 and 30 and for various $m$.}\label{tab5}
	\begin{tabular}{c c c c c c c c c c c c}
		\hline
		& \multicolumn{2}{c}{$\hat{\xi}_{\alpha}^w(X)$}&\multicolumn{2}{c}{$\xi_{\alpha}^wL$} & & \multicolumn{2}{c}{$\xi_{\alpha}^wV$}&\multicolumn{2}{c}{$\xi_{\alpha}^wE$} & \multicolumn{2}{c}{$\xi_{\alpha}^wN$}\\
		\hline
		$n$ & Bias & MSE & Bias & MSE & $m$ & Bias & MSE & Bias & MSE & Bias & MSE\\
		\hline
	    10 & -0.0086 & \bf0.00032 & -0.0086 & \bf0.00043 & 1 & -0.0152 & \bf0.00050 & -0.0145 & 0.00048 & -0.0127 & 0.00043\\
		& & & & & 2 & -0.0171 & 0.00052 & -0.0137 & 0.00043 & -0.0086 & 0.00034\\
		& & & & & 3 & -0.0200 & 0.00058 & -0.0129 & 0.00039 & -0.0025 & \bf0.00028\\
		& & & & & 4 & -0.0236 & 0.00070 & -0.0119 & \bf0.00035 & 0.0047 &0.00034\\[1ex]
		
		20 & -0.0043 & \bf0.00014 & 0.0041 & \bf0.00017 & 1 & -0.0077 & \bf0.00019 & -0.0078 & 0.00020 & -0.0076 & 0.00019\\
		& & & & & 2 & -0.0085 & 0.00020 & -0.0077 & 0.00018 & -0.0065 & 0.00017\\
		& & & & & 3 & -0.0090 & 0.00020 & -0.0075 & 0.00017 & -0.0049 & 0.00014\\
		& & & & & 4 & -0.0104 & 0.00022 & -0.0073 & 0.00016 & -0.0029 & 0.00011\\
		& & & & & 5 & -0.0118 & 0.00023 & -0.0070 & 0.00015 & -0.00035 & \bf0.00010\\
		& & & & & 6 & -0.0135 & 0.00027 & -0.0070 & 0.00014 & 0.0026 & \bf0.00010\\
		& & & & & 7 & -0.0154 & 0.00031 & -0.0067 & 0.00013 & 0.0056 & 0.00014\\
		& & & & & 8 & -0.0175 & 0.00036 & -0.0064 & 0.00012 & 0.0092 & 0.00020\\
		& & & & & 9 & -0.0197 & 0.00044 & -0.0061 & \bf0.00011 & 0.0130 & 0.00029\\[1ex]
		
		30 & -0.0029 & \bf0.00009 & 0.0028 & \bf0.00010 & 1 & -0.0053 & \bf0.00011 & -0.0052 & 0.00011 & -0.0051 & 0.00011\\
		& & & & & 2 & -0.0055 & \bf0.00011 & -0.0052 & 0.00011 & -0.0046 & 0.00010\\
		& & & & & 3 & -0.0061 & 0.00012 & -0.0051 & 0.00011 & -0.0041 & 0.00009\\
		& & & & & 4 & -0.0066 & 0.00012 & -0.0051 & 0.00010 & -0.0031 & 0.00008\\
		& & & & & 5 & -0.0071 & 0.00012 & -0.0050 & 0.00009 & -0.0020 & 0.00007\\
		& & & & & 6 & -0.0081 & 0.00013 & -0.0048 & 0.00008 & -0.0007 & \bf0.00006\\
		& & & & & 7 & -0.0089 & 0.00014 & -0.0050 & 0.00008 & 0.0008 & \bf0.00006\\
		& & & & & 8 & -0.0100 & 0.00015 & -0.0048 & 0.00008 & 0.0025 & \bf0.00006\\
		& & & & & 9 & -0.0112 & 0.00017 & -0.0048 & 0.00008 & 0.0045 & 0.00008\\
		& & & & & 10& -0.0126 & 0.00020 & -0.0046 & 0.00007 & 0.0065 & 0.00010\\
		& & & & & 11& -0.0140 & 0.00023 & -0.0045 & 0.00007 & 0.0086 & 0.00013\\
		& & & & & 12& -0.0154 & 0.00027 & -0.0043 & \bf0.00006 & 0.0110 & 0.00018\\
		& & & & & 13& -0.0167 & 0.00031 & -0.0041 & \bf0.00006 & 0.0133 & 0.00024\\
		& & & & & 14& -0.0184 & 0.00036 & -0.0043 & \bf0.00006 & 0.0160 & 0.00032\\\hline  
	\end{tabular}
\end{table}
\end{landscape}

\begin{landscape}
\begin{table}[h!]
	\centering
	\caption{Bias and MSE of the estimators for WE(2,1) distribution with $\alpha$ = 2, $n$ = 10, 20 and 30 and for various $m$.}\label{tab6}
	\begin{tabular}{c c c c c c c c c c c c}
		\hline
		& \multicolumn{2}{c}{$\hat{\xi}_{\alpha}^w(X)$}&\multicolumn{2}{c}{$\xi_{\alpha}^wL$} & & \multicolumn{2}{c}{$\xi_{\alpha}^wV$}&\multicolumn{2}{c}{$\xi_{\alpha}^wE$} & \multicolumn{2}{c}{$\xi_{\alpha}^wN$}\\
		\hline
		$n$ & Bias & MSE & Bias & MSE & $m$ & Bias & MSE & Bias & MSE & Bias & MSE\\
		\hline
		10 & -0.0250 & \bf0.0078 & 0.0251 & \bf0.0102 & 1 & -0.0582 & \bf0.0081 & -0.0562 & 0.0078 & -0.0490 & 0.0072\\
		& & & & & 2 & -0.0588 & \bf0.0081 & -0.0452 & \bf0.0073 & -0.0274 & \bf0.0068\\
		& & & & & 3 & -0.0612 & 0.0086 & -0.0381 & 0.0074 & -0.0038 & 0.0080\\
		& & & & & 4 & -0.0672 & 0.0092 & -0.0305 & 0.0077 & 0.0207 & 0.0107\\[1ex]
		
		20 & -0.0120 & \bf0.0040 & 0.0116 & \bf0.0046 & 1 & -0.0356 & 0.0042 & -0.0347 & 0.0042 & -0.0327 & 0.0040\\
		& & & & & 2 & -0.0328 & 0.0042 & -0.0277 & \bf0.0039 & -0.0210 & \bf0.0038\\
		& & & & & 3 & -0.0310 & \bf0.0041 & -0.0206 & \bf0.0039 & -0.0069 & 0.0040\\
		& & & & & 4 & -0.0295 & \bf0.0041 & -0.0145 & 0.0041 & -0.0062 & 0.0047\\
		& & & & & 5 & -0.0296 & 0.0043 & -0.0079 & 0.0043 & -0.0195 & 0.0059\\
		& & & & & 6 & -0.0285 & 0.0043 & -0.0023 & 0.0047 & -0.0343 & 0.0074\\
		& & & & & 7 & -0.0300 & 0.0047 & -0.0030 & 0.0052 & -0.0490 & 0.0097\\
		& & & & & 8 & -0.0310 & 0.0049 & 0.0078 & 0.0055 & -0.0634 & 0.0126\\
		& & & & & 9 & -0.0328 & 0.0052 & 0.0120 & 0.0062 & -0.0780 & 0.0154\\[1ex]
		
		30 & -0.0081 & \bf0.0027 & 0.0074 & \bf0.0030 & 1 & -0.0268 & 0.0029 & -0.0260 & 0.0028 & -0.0251 & 0.0028\\
		& & & & & 2 & -0.0234 & \bf0.0027 & -0.0207 & \bf0.0027 & -0.0165 & \bf0.0026\\
		& & & & & 3 & -0.0235 & 0.0028 & -0.0164 & \bf0.0027 & -0.0080 & 0.0027\\
		& & & & & 4 & -0.0194 & \bf0.0027 & -0.0099 & \bf0.0027 & 0.0035 & 0.0031\\
		& & & & & 5 & -0.0174 & \bf0.0027 & -0.0044 & 0.0030 & 0.0130 & 0.0038\\
		& & & & & 6 & -0.0155 & 0.0028 & 0.0009 & 0.0031 & 0.0240 & 0.0045\\
		& & & & & 7 & -0.0139 & 0.0030 & 0.0064 & 0.0035 & 0.0345 & 0.0058\\
		& & & & & 8 & -0.0177 & 0.0032 & 0.0114 & 0.0040 & 0.0453 & 0.0073\\
		& & & & & 9 & -0.0119 & 0.0033 & 0.0148 & 0.0043 & 0.0577 & 0.0092\\
		& & & & & 10& -0.0114 & 0.0034 & 0.0210 & 0.0049 & 0.0676 & 0.0111\\
		& & & & & 11& -0.0122 & 0.0036 & 0.0248 & 0.0054 & 0.0780 & 0.0132\\
		& & & & & 12& -0.0117 & 0.0037 & 0.0288 & 0.0057 & 0.0877 & 0.0154\\
		& & & & & 13& -0.0138 & 0.0038 & 0.0335 & 0.0066 & 0.0985 & 0.0180\\
		& & & & & 14& -0.0149 & 0.0040 & 0.0359 & 0.0069 & 0.1104 & 0.0213\\\hline  
	\end{tabular}
\end{table}
\end{landscape}

\section{Estimation of WCRTE for Progressively Type-II Censored Data}\label{S5} So far, we have discussed non-parametric estimation of WCRTE measure for complete sample. In this section, we propose an estimator for WCRTE measure for progressively type-II censored data. Censoring plays a crucial role in reliability and life-testing as it is not always possible to conduct experiments till all the units fail due to time and cost constraints. Type-I, type-II and hybrid censoring are commonly used in life-testing but these censoring schemes do not allow removal of items during the experiment. Cohen (1963) introduced progressive type-II (PC-II) censoring scheme which allows removal of units during the experiment. It is a generalization of the type-II censoring. A PC-II censoring can be described as follows:\\

Suppose $n$ identical units are put to a test, each having a common lifetime distribution with pdf $f$ and cdf $F$. Also, suppose a pre-fixed number $r$ of failures is allowed. Let $R_1,\ldots, R_r$ be prefixed integers such that $R_1 +\cdots+R_r = n-r$. When the first failure occurs, $R_1$ of the remaining $n-1$ surviving units are randomly removed from the test. Then $R_2$ of the remaining $n-R_1-2$ units are randomly removed at the time of second failure. Proceeding this way, at the time of the $r$th failure, all the remaining $R_r = n-r-\sum_{i=1}^{r-1}R_i$ units are removed from the test. The failure times $X_{1:r:n},\cdots,X_{r:r:n}$ are called progressive type-II censored order statistics. PC-II censoring reduces to the conventional type-II censoring when $R_i = 0$ for $i=1, \ldots, r-1$, and $R_r=n-r$.\\

Let $F_{r:n}(x)$ denote the empirical distribution function of the PC-II censored data. Balakrishnan and Sandhu (1995) expressed $F_{r:n}(x)$ as

$$F_{r:n}(x)=
\begin{cases}
0,\;\;\;\;\mbox{if}\;x<X_{1:r:n},\\
\alpha_{i:r:n},\;\;\;\mbox{if}\;X_{i:r:n}\leq x<X_{i+1:r:n},\;i=1,2,\cdots r-1,\\
\alpha_{r:r:n},\;\;\;\mbox{if}\;x \geq X_{r:r:n},
\end{cases}$$
where $\alpha_{i:r:n}=E(U_{i:r:n})$ and $U_{i:r:n}$ is the $i$th progressive type-II censored order statistic from U(0,1) distribution. Note that $$E(U_{i:r:n})=1-\prod_{k=r-i+1}^{r}\beta_k,$$ where $\beta_i=\frac{i+\sum_{k=r-i+1}^{r}R_k}{1+i+\sum_{k=r-i+1}^{r}R_k}$, $\beta_k=\beta_1$ if $k\leq 1$ and $\beta_k=\beta_r$ if $k\geq r$. For detailed discussions on PC-II censoring, readers may refer to Balakrishnan and Aggarwala (2000) and Balakrishnan and Cramer (2014).\\

Suppose $n$ items are put to a PC-II life-test and $X_{1:r:n},\cdots,X_{r:r:n}$ are the corresponding progressive type-II order statistics. We propose an estimator for WCRTE as 

\begin{eqnarray}\label{eN}
\xi_{\alpha}^wP&=&\frac{1}{(\alpha-1)}\int_{0}^{X_{i:r:n}}x\left(\bar{F}_{r:n}(x)-\bar{F}^{\alpha}_{r:n}(x) \right)dx\nonumber\\ 
&=&\frac{1}{2(\alpha-1)}\sum_{i=1}^{r}\left( X^2_{i+1:r:n}-X^2_{i:r:n}\right) \left[ \left( 1-E(U_{i:r:n})\right) -\left( 1-E(U_{i:r:n})\right)^{\alpha}\right].
\end{eqnarray}
\begin{p}\label{P1}
	The estimator $\xi_{\alpha}^wP$ is consistent when $r\to n$, $n\to \infty$ and $\alpha>1$.
\end{p}

\begin{proof}
	When $r\to n$, the PC-II censored sample becomes complete sample and the required result follows from Lemma \ref{l2}.
\end{proof}

\begin{p}
	For type-II censoring at $r$, the estimator reduces to $$\xi_{\alpha}^wP=\frac{1}{2(\alpha-1)}\sum_{i=1}^{r}\left( X^2_{i+1:r:n}-X^2_{i:r:n}\right)\left[\left(1-\frac{i}{n+1}\right) -\left(1-\frac{i}{n+1}\right)^{\alpha} \right].$$
\end{p}
\begin{proof}
	The PC-II censoring becomes type-II censoring when $R_i = 0$ for $i=1, \ldots, r-1$, and $R_r=n-r$. Now $\beta_i=\frac{i+n-r}{i+1+n-r}$ for all $i$ and $\prod_{k=r-i+1}^{r}\beta_k=\frac{n-i+1}{n+1}$. Substituting this in (\ref{eN}) we get the result.
	\end{proof}
		
\begin{table}[h!]
	\centering
	\caption{Average Estimate and Variance of $\xi_{\alpha}^wP$ for U(0,1) and Exp(1) distributions.}\label{TN1}
	\begin{tabular}{p{1cm} p{1cm} p{3cm}| p{2cm} p{2cm}| p{2cm} p{1.5cm}}
		\hline
		& & & \multicolumn{2}{c|}{U(0,1)}&\multicolumn{2}{c}{Exp(1)}\\
		\hline
		$n$ & $r$ & Schemes & AE & Var & AE & Var\\
		\hline
		20 & 10 & ($10,0*9$)   & 0.08236 & 0.00026 & 0.88841 & 0.61174\\
		&    & ($0*9,10$)   & 0.02271 & 0.00010 & 0.04970  & 0.00118\\
		&    & ($5,0*8,5$)  & 0.04188 & 0.00025 & 0.13144 & 0.00848\\
		&    & ($1*10$)     & 0.06035 & 0.00050 & 0.29540 & 0.07601\\[1ex]
		& 15 & ($5,0*14$)   & 0.08470 & 0.00015 & 0.93225 & 0.44237\\
		&    & ($0*14,5$)   & 0.05593 & 0.00022 & 0.21008 & 0.01470\\
		&    &($0*7,5,0*7$) & 0.08487 & 0.00019 & 0.89215 & 0.50629\\
		&    & ($3,0*13,2$) & 0.07328 & 0.00022 & 0.41973 & 0.06648\\[1ex]
		30 & 10 & ($20,0*9$)   & 0.08228 & 0.00027 & 0.87797 & 0.60185\\
		&    & ($0*9,20$)   & 0.00824 & 0.00002 & 0.01364 & 0.00009\\
		&    &($10,0*8,10$) & 0.02180 & 0.00010 & 0.04666 & 0.00105\\
		&    & ($2*10$)     & 0.03872 & 0.00038 & 0.12773 & 0.01706\\
		& 20 & ($10,0*19$)  & 0.08501 & 0.00011 & 0.92785 & 0.34480\\
		&    & ($0*19,10$)  & 0.04610 & 0.00014 & 0.13597 & 0.00458\\
		&    & ($5,0*18,5$) & 0.06410 & 0.00018 & 0.27141 & 0.01824\\[1ex]
		50 & 30 & ($20,0*29$)  & 0.08500 & 0.00007 & 0.91889 & 0.22099\\
		&    & ($0*29,20$)  & 0.03784 & 0.00008 & 0.09597 & 0.00142\\
		&    &($10,0*28,10$)& 0.05880 & 0.00012 & 0.20955 & 0.00695\\\hline
	\end{tabular}
\end{table}
		
\subsection{Simulation Study} To assess the performance of the proposed estimator we use Monte-Carlo simulation. We generate 5000 PC-II censored samples from U(0,1) and Exp(1) distributions and calculate average estimate (AE) and variance of $\xi_{\alpha}^wP$. We consider $(n,r)$ = (20,10), (20,15), (30,10), (30,20) and (50,30) and $\alpha$ = 2. The results are presented in Table \ref{TN1} for various progressive type-II censoring schemes. In the table, the notation $a*b$ refers to `$a$ is repeated $b$ times'. For instance, $(5,0*8,5)$ refers to $(5,0,0,0,0,0,0,0,0,5)$. From the table, it is observed that when the sample size increases the variance of the estimator decreases for both the distributions. The overall performance of the estimator is good for both U(0,1) and Exp(1) distributions. As we have seen for complete samples, for PC-II censored data also, the estimator performs better for U(0,1) than Exp(1) distribution.\\

Chakraborty and Pradhan (2023) showed that WCRTE can be used as a risk (volatility) measure and they compared its performance with other popular risk measures such as standard deviation, right tail risk (Wang, 1998) using stock market (BSE SENSEX) data. Among all progressive censoring schemes, type-II censoring has the minimum volatility because it contains maximum information. From Table \ref{TN1}, we can see that for type-II censoring schemes the value of the AE of WCRTE is minimum and the corresponding variance of the estimator is also minimum for both the distributions. As the AE value increases, the variance of the estimator also increases. The PC-II censoring scheme where all the remaining items are removed after the first failure is called type-III censoring scheme (Pradhan and Kundu, 2009). Type-III censoring is opposite of type-II censoring and it has maximum volatility (minimum information) among all schemes. The AE values for type-III censoring are maximum, resulting in higher variance of the estimator $\xi_{\alpha}^wP$.

\section{Applications}\label{S6} In this section, we develop a goodness-of-fit test for uniform distribution using an estimator of WCRTE measure for complete and progressively type-II censored data. Uniform distribution is used for random number generation. Uniformity test is an important problem in statistics and this problem has been addressed by many researchers in the literature. For example, Stephens (1974) studied uniformity test using cdf-based statistic and Dudewicz and Van Der Meulen (1981) proposed an entropy-based test of uniformity. For recent development in this area, redears may refer to Noughabi (2022), Chakraborty et al. (2023) and the references therin.\\

\subsection{Uniformity Test for Complete Sample}
Let $X_1,X_2,\cdots,X_n$ be a random sample from a continuous distribution with pdf $f$ concentrated on [0,1], i.e. $f(x)=0$ if $x\notin$[0,1]. Suppose $X_{(1)},X_{(2)},\cdots,X_{(n)}$ are the corresponding order statistics. We want to test whether the data come from a U(0,1) distribution. So we want to test
\begin{eqnarray}
H_0:f(x)\sim U(0,1)\;\;\;\;vs.\;\;\;\;H_1:f(x)\not\sim U(0,1).\nonumber
\end{eqnarray}
Consider the estimator of WCRTE $$\hat{\xi}_{\alpha}^w(X)=\frac{1}{2(\alpha-1)}\sum_{i=1}^{n-1}\left( X^2_{(i+1)}-X^2_{(i)}\right) \left[ \left( 1-\frac{i}{n}\right) -\left( 1-\frac{i}{n}\right)^{\alpha}\right].$$ This estimator was introduced by Chakraborty and Pradhan (2023). We will use this estimator as a test statistic. Note that, when $\alpha\to1$, $\hat{\xi}_{\alpha}^w(X)$ reduces to $\hat{\xi}^w(X)$ where $\hat{\xi}^w(X)$ is the estimator of WCRE (Misagh et al., 2011).

\noindent The following lemma will be useful for the construction of the test.

\begin{lem}\label{l1}
	Let $X_1,X_2,\cdots,X_n$ be a random sample drawn from a continuous distribution concentrated on [0,1]. Then, $0\leq\hat{\xi}_{\alpha}^w(X)\leq\frac{1}{2\alpha^{\frac{\alpha}{\alpha-1}}}$ for $\alpha>1$.
\end{lem}

\begin{proof}
	Note that, $f(v)=(1-v)-(1-v)^{\alpha},\,0<v<1$, has the maximum value $\frac{\alpha-1}{\alpha^{\frac{\alpha}{\alpha-1}}}$ for $\alpha>1$.  So we have,
	$$\hat{\xi}^w_{\alpha}(X)\leq\frac{1}{2(\alpha-1)}\frac{(\alpha-1)}{\alpha^{\frac{\alpha}{\alpha-1}}}\sum_{i=1}^{n-1}(X^2_{(i+1)}-X^2_{(i)})=\frac{1}{2\alpha^{\frac{\alpha}{\alpha-1}}}(X^2_{(n)}-X^2_{(1)})\leq\frac{1}{2\alpha^{\frac{\alpha}{\alpha-1}}}.$$
	Hence the result follows.
\end{proof}

Under $H_0$, $\xi_{\alpha}^w(X)=\frac{(\alpha+4)}{6(\alpha+1)(\alpha+2)}$, which lies between $\left(0,\frac{1}{2\alpha^{\frac{\alpha}{\alpha-1}}}\right)$. Therefore, we can use $\hat{\xi}_{\alpha}^w(X)$ as a test statistic and we will reject the null hypothesis when $\hat{\xi}_{\alpha}^w(X)$ is large or small.The critical region for a sample of size $n$ and significance level $\gamma$ is $$ \hat{\xi}_{\alpha}^w(X)\leq\mathcal{C}_{\frac{\gamma}{2},n}\;\;or\;\;\hat{\xi}_{\alpha}^w(X)\geq\mathcal{C}_{1-\frac{\gamma}{2},n},$$ where $\mathcal{C}_{\gamma,n}(\mathcal{C}_{1-\gamma,n})$ is the lower (upper) quantile point of the empirical distribution of $\hat{\xi}_{\alpha}^w(X)$. Note that, when $\alpha=6.586506$, the value of $\xi_{\alpha}^w(X)$ under $H_0$ lies in the middle of $\left(0,\frac{1}{2\alpha^{\frac{\alpha}{\alpha-1}}}\right)$.

\begin{table}[h!]
	\centering
	\caption{Critical values of $\hat{\xi}^w_{\alpha}(X)$.}\label{tab7}
	\begin{tabular}{p{.5cm} p{1.5cm} p{2cm} p{2cm} p{.5cm} p{1.5cm} p{2cm} p{2cm}}
		\hline
		$n$ & $\alpha$ & $\mathcal{C}_{0.025,n}$ & $\mathcal{C}_{0.975,n}$ & $n$ & $\alpha$ & $\mathcal{C}_{0.025,n}$ & $\mathcal{C}_{0.975,n}$\\
		\hline
		10 & 1 & 0.06755 & 0.15720 & 15 & 1 & 0.08561 & 0.15572\\
		& 2 & 0.04104 & 0.10149 &    & 2 & 0.05006 & 0.09995\\
		& 5 & 0.01700 & 0.04822 &    & 5 & 0.02056 & 0.04663\\
		& 7 & 0.01175 & 0.03553 & &    7 & 0.01431 & 0.03407\\
		& 10 & 0.00823 & 0.02545 & & 10 & 0.00992 & 0.02442\\[1ex]
		
		20 & 1 & 0.09670 & 0.15478 & 25 & 1 & 0.10236 & 0.15404\\
		& 2 & 0.05553 & 0.09875 &    & 2 & 0.05888 & 0.09742\\
		& 5 & 0.02248 & 0.04581 &    & 5 & 0.02383 & 0.04495\\
		& 7 & 0.01584 & 0.03327 & &    7 & 0.01691 & 0.03273\\
		& 10 & 0.01093 & 0.02360 & & 10 & 0.01170 & 0.02300\\[1ex] 
		
		30 & 1 & 0.10791 & 0.15291 & 35 & 1 & 0.10948 & 0.15209\\
		& 2 & 0.06160 & 0.09629 &    & 2 & 0.06362 & 0.09556\\
		& 5 & 0.02508 & 0.04392 &    & 5 & 0.02594 & 0.04355\\
		& 7 & 0.01780 & 0.03209 & &    7 & 0.01840 & 0.03158\\
		& 10 & 0.01215 & 0.02252 & & 10 & 0.01265 & 0.02226\\[1ex]
		
		40 & 1 & 0.11240 & 0.15153 & 45 & 1 & 0.11442 & 0.15112\\
		& 2 & 0.06460 & 0.09505 &    & 2 & 0.06611 & 0.09450\\
		& 5 & 0.02643 & 0.04316 &    & 5 & 0.02730 & 0.04286\\
		& 7 & 0.01883 & 0.03137 & &    7 & 0.01925 & 0.03116\\
		& 10 & 0.01294 & 0.02202 & & 10 & 0.01320 & 0.02180\\[1ex] 
		
		50 & 1 & 0.11610 & 0.15061 & 60 & 1 & 0.11894 & 0.15022\\
		& 2 & 0.06712 & 0.09420 &    & 2 & 0.06870 & 0.09370\\
		& 5 & 0.02758 & 0.04248 &    & 5 & 0.02850 & 0.04196\\
		& 7 & 0.01940 & 0.03073 & &    7 & 0.02021 & 0.03044\\
		& 10 & 0.01347 & 0.02167 & & 10 & 0.01392 & 0.02121\\[1ex]
		
		70 & 1 & 0.12029 & 0.14941 & 80 & 1 & 0.12218 & 0.14885\\
		& 2 & 0.06980 & 0.09252 &    & 2 & 0.07075 & 0.09232\\
		& 5 & 0.02884 & 0.04162 &    & 5 & 0.02927 & 0.04134\\
		& 7 & 0.02051 & 0.03004 & &    7 & 0.02073 & 0.02971\\
		& 10 & 0.01416 & 0.02106 & & 10 & 0.01429 & 0.02082\\[1ex]
		
		90 & 1 & 0.12325 & 0.14842 & 100 & 1 & 0.12441 & 0.14791\\
		& 2 & 0.07165 & 0.09202 &    & 2 & 0.07242 & 0.09160\\
		& 5 & 0.02974 & 0.04096 &    & 5 & 0.02999 & 0.04071\\
		& 7 & 0.02103 & 0.02955 & &    7 & 0.02119 & 0.02926\\
		& 10 & 0.01448 & 0.02064 & & 10 & 0.01474 & 0.02050\\
		\hline
	\end{tabular}
\end{table}

\begin{lem}\label{l2}
	The test based on $\hat{\xi}_{\alpha}^w(X)$ is consistent.
\end{lem}

\begin{proof}
	We have, 
	\begin{eqnarray}\label{e23}
	\hat{\xi}_{\alpha}^w(X)&=&\frac{1}{\alpha-1}\int_{0}^{+\infty}x\left(\bar{F}_n(x)-\bar{F}^{\alpha}_n(x) \right)dx\nonumber\\
	&=&\frac{1}{\alpha-1}\left[ \int_{0}^{1}x\left(\bar{F}_n(x)-\bar{F}^{\alpha}_n(x) \right)dx+\int_{1}^{+\infty}x\left(\bar{F}_n(x)-\bar{F}^{\alpha}_n(x)\right)dx\right].
	\end{eqnarray}
	Using Glivenco-Cantelli theorem, the first integral in Eq. (\ref{e23}) converges to\\ $\int_{0}^{1}x(\bar{F}(x)-\bar{F}^{\alpha}(x))dx$.\\ We have from Rao et al. (2004), for $x\in[1,+\infty)$,
	$$\bar{F}_n(x)\leq x^{-p}\left(\sup_n \frac{1}{n}\sum_{i=1}^{n}X_i^p\right).$$
	
	So, $\int_{1}^{+\infty}x\bar{F}_n(x)dx\leq\left(\sup_n \frac{1}{n}\sum_{i=1}^{n}X_i^p\right)\int_{1}^{+\infty}x^{1-p}dx$, converges for $p>2$. Similarly, $\int_{1}^{+\infty}x\bar{F}^{\alpha}_n(x)dx$ converges for $p>\frac{2}{\alpha}.$\\
	Therefore, $\hat{\xi}_{\alpha}^w(X)$ is a consistent estimator. This implies $\hat{\xi}_{\alpha}^w(X)\to \frac{(\alpha+4)}{6(\alpha+1)(\alpha+2)}$ under $H_0$. Hence the result.
\end{proof}

We can also develop uniformity test using $\hat{\xi}^w(X)$ which is given by
$$\hat{\xi}^w(X)=-\frac{1}{2}\sum_{i=1}^{n-1}\left( X^2_{(i+1)}-X^2_{(i)}\right)\left( 1-\frac{i}{n}\right)\log \left(  1-\frac{i}{n}\right) .$$
\begin{lem}
		Let $X_1,X_2,\cdots,X_n$ be a random sample drawn from a continuous distribution concentrated on [0,1]. Then, $0\leq\hat{\xi}^w(X)\leq\frac{1}{2e}$.
\end{lem}
\begin{proof}
	The function $g(v)=-(1-v)\log(1-v),\;0<v<1,$ has maximum value $\frac{1}{2}$. Proceeding along the same line as in Lemma \ref{l1}, we can get the desired result.
\end{proof}

Under $H_0$, $\xi^w(X)=\frac{5}{36}=0.1389$ which lies between $(0,\frac{1}{2e})$. Note that, $\frac{1}{2e}=0.1839$. So we can use $\hat{\xi}^w(X)$ as a test statistic for testing uniformity and we will reject the null hypothesis when either $\hat{\xi}^w(X)<C_1$ or $\hat{\xi}^w(X)>C_2$, where $C_1$ and $C_2$ will be calculated from the size condition. On using Theorem 7.1 of Mirali et al. (2017) we can see that, $\hat{\xi}^w(X)$ is consistent. So the uniformity test based on $\hat{\xi}^w(X)$ is also consistent.

\begin{table}[h!]
	\centering
	\caption{Power of the tests for various sample sizes.} \label{tab8}
	\resizebox{16cm}{!}{
		\begin{tabular}{c c c c c c c c c c c}
			\hline
			$n$ & Alternatives & $\alpha=1$ & $\alpha=2$ & $\alpha=5$ & $\alpha=7$ & $\alpha=10$ & ENT & KS & CvM & AD \\[1ex]
			\hline
			10 & $A_{1.5}$ & 0.1400 & 0.1742 & \bf0.1905 & 0.1683 & 0.1860 & 0.1420 & 0.1489 & 0.1660 & 0.1608\\
			& $A_{2}$      & 0.3455 & 0.4228 & 0.4675 & 0.4609 & \bf0.4740 & 0.2893 & 0.3762 & 0.4378 & 0.4091  \\
			& $B_{1.5}$    & 0.0915 & 0.0865 & 0.0581 & 0.0456 & 0.0447 & \bf0.1884 & 0.0421 & 0.0330 & 0.0200\\
			& $B_{2}$      & 0.2263 & 0.2055 & 0.1209 & 0.0907 & 0.0804 & \bf0.4363 & 0.0398 & 0.0226 & 0.0081\\
			& $B_{3}$      & 0.5612 & 0.5549 & 0.3974 & 0.2917 & 0.2526 & \bf0.7958 & 0.0908 & 0.0486 & 0.0182\\
			& $C_{1.5}$    & 0.1004 & 0.1191 & 0.1310 & 0.1222 & 0.1063 & 0.0302 & 0.1137 & 0.0978 & \bf0.1354\\
			& $C_{2}$      & 0.1807 & 0.2184 & \bf0.2343 & 0.2063 & 0.1911 & 0.0330 & 0.2003 & 0.1491 & 0.2180\\[1ex]
			
			20 & $A_{1.5}$ & 0.2682 & 0.3230 & 0.3443 & 0.3468 & \bf0.3543 & 0.2411 & 0.2841 & 0.3184 & 0.3065\\
			& $A_{2}$      & 0.6483 & 0.7555 & 0.8063 & 0.8105 & \bf0.8163 & 0.6207 & 0.7038 & 0.7726 & 0.7480\\
			& $B_{1.5}$    & 0.2217 & 0.1739 & 0.0964 & 0.0762 & 0.0674 & \bf0.3039 & 0.0555 & 0.0502 & 0.0258\\
			& $B_{2}$      & 0.5753 & 0.5012 & 0.2839 & 0.2014 & 0.1627 & \bf0.7115 & 0.1176 & 0.0972 & 0.1003\\
			& $B_{3}$      & 0.9615 & 0.9466 & 0.8079 & 0.7058 & 0.5753 & \bf0.9905 & 0.4201 & 0.5089 & 0.5639\\
			& $C_{1.5}$    & 0.1477 & \bf0.1732 & 0.1603 & 0.1516 & 0.1308 & 0.0589 & 0.1455 & 0.1224 & 0.1633\\
			& $C_{2}$      & 0.3210 & \bf0.3897 & 0.3153 & 0.2806 & 0.2313 & 0.1461 & 0.3092 & 0.2538&0.3846\\
			[1ex]		
			
			30 & $A_{1.5}$ & 0.3971 & 0.4644 & 0.5068 & \bf0.5357  & 0.5201 & 0.3268 & 0.3988 & 0.4713 & 0.4667\\
			& $A_{2}$      & 0.8451 & 0.9204 & 0.9468 & \bf0.9506  & 0.9453 & 0.8201 & 0.8617 & 0.9192 & 0.9173  \\
			& $B_{1.5}$    & 0.3567 & 0.2768 & 0.1495 & 0.1202  & 0.0838 & \bf0.4047 & 0.0788 & 0.0585 & 0.0600\\
			& $B_{2}$      & 0.8232 & 0.7506 & 0.4788 & 0.3880  & 0.2799 &\bf0.8849 & 0.2413 & 0.8044 & 0.3011\\
			& $B_{3}$      & 0.9989 & 0.9968 & 0.9653 & 0.9263  & 0.8299 & \bf0.9991 & 0.7320 & 0.9969 & 0.9259\\
			& $C_{1.5}$    & 0.2061 & \bf0.2567 & 0.2233 & 0.1803  & 0.1594 & 0.1112 & 0.1836 & 0.3184 & 0.2022\\
			& $C_{2}$      & 0.4695 & \bf0.5798 & 0.4183 & 0.3313  & 0.2837 & 0.3577 & 0.4403 & 0.4016 & 0.5296\\
			[1ex]			
			\hline
	\end{tabular}}
\end{table}

\subsubsection{Critical Points and Power of the Tests}
The exact distributions of $\hat{\xi}_{\alpha}^w(X)$ and $\hat{\xi}^w(X)$ are intractable. We use simulations to obtain the critical points for various choices of $\alpha$ and $n$. We generate 10000 random samples from U(0,1) distribution with different sample sizes and calculate the critical points at 5\% level of significance. The critical points are provided in Table \ref{tab7}. In the table, by $\alpha=1$ we mean $\alpha\to1$ i.e WCRE. We take $\alpha$ = 1, 2, 5, 7 and 10, and calculate the quantile points for $n$ = 10(5)50 and 50(10)100. For power calculation, we consider the following seven alternative distributions. The cdfs of these distributions are as follows:
\begin{eqnarray}
A_j:F(z)&=&1-(1-z)^j,\;\;\;0\leq z\leq1\;(j=1.5,\;2)\nonumber\\[1ex]
B_j:F(z)&=&\begin{cases}
2^{j-1}z^j,\;\;\;\;\;\;\;\;\;\;\;\;\;\;\;\;\;0\leq z\leq 0.5\\
\;\;\;\;\;\;\;\;\;\;\;\;\;\;\;\;\;\;\;\;\;\;\;\;\;\;\;\;\;\;\;\;\;\;\;\;\;\;\;\;\;\;\;\;\;\;\;(j=1.5,\;2,\;3)\\
1-2^{j-1}(1-z)^j,\,\,0.5\leq z\leq 1 
\end{cases}\nonumber\\[1ex]
C_j:F(z)&=&\begin{cases}
0.5-2^{j-1}(0.5-z)^j,\;\;0\leq z\leq 0.5\\
\;\;\;\;\;\;\;\;\;\;\;\;\;\;\;\;\;\;\;\;\;\;\;\;\;\;\;\;\;\;\;\;\;\;\;\;\;\;\;\;\;\;\;\;\;\;\;\;\;\;\;\;\;(j=1.5,\;2)\\
0.5+2^{j-1}(z-0.5)^j,\,\,0.5\leq z\leq 1 	
\end{cases}\nonumber
\end{eqnarray} 

These alternative distributions were first considered by Stephens (1974). The alternatives $A_j$'s provide points close to 0 than expected under uniformity and it is interpreted as a change in mean, the alternatives $B_j$'s provide points close to the middle i.e. 0.5 and it is interpreted as a change towards smaller variance, and finally, the alternatives $C_j$'s provide points towards both the extremes viz. 0 and 1, and it is interpreted as a shift towards larger variance. We compare power of our proposed tests with Kolmogorov-Smirnov (KS), Cramer-von Mises (CvM), Anderson-Darling (AD) tests and Vasicek sample entropy-based test of Dudewicz and Van Der Meulen (1981). The test statistic is defined as $$H_{mn}=\frac{1}{n}\sum_{k=1}^{n}\log[\frac{n}{2m}(X_{(k+m)}-X_{(k-m)})],$$ where $m$ is the window size. Under $H_0$, $H_{mn}$ converges to zero and it is less than zero otherwise. So, we will reject $H_0$ when $H_{mn}$ is small. This test is called the ENT test. We compute the power for $n$ = 10, 20 and 30, and present the findings in Table \ref{tab8}. From Table \ref{tab8}\\ we observe that, ENT test of Dudewicz and Van Der Meulen (1981) performs better than other tests for alternatives $B_j$'s. Our proposed tests perform better than all the tests for alternatives $A_j$'s and $C_j$'s. When $\alpha$ = 2, proposed test performs better for alternative $C_j$'s and for higher values of $\alpha$ the proposed tests perform better for alternative $A_j$'s. Also note that, the WCRTE-based tests perfrom better than the test based on WCRE ($\alpha\to$1). The alternatives $A_j$'s and $C_j$'s provides points to the extreme. Our tests are based on weighted measures that not only consider the probabilistic information of the random variables but also the values that the random variables are taking. Thus, they are less influenced by the extreme observations. Therefore, we can rely on these measures when inference has to be made based on extreme observations and tail probabilities.

\pagebreak

\subsection{Uniformity under Progressively Type-II Censored Data}
Now we extend the uniformity test for complete sample to progressively type-II censored sample. Let $X_{1:r:n},\cdots,X_{r:r:n}$ be a progressively type-II censored sample with censoring scheme $(R_1,R_2,\cdots,R_r)$ from a continuous rv having density $f$ concentrated on [0,1]. Based on this sample, we want to test whether $f$ is uniformly distributed or not. We consider the estimator $\xi_{\alpha}^wP$ with $\alpha$ = 2. From Proposition \ref{P1}, we found that $\xi_{\alpha}^wP$ is a consistent estimator when $r\to n$ i.e. $\xi_{\alpha}^wP$ reduces to $\hat{\xi}_{\alpha}^w(X)$ as $r\to n$. From Lemma \ref{l2} we have that $\hat{\xi}_{\alpha}^w(X)$ lies between (0,0.125) for $\alpha$ = 2. So, we can use $\xi_{\alpha}^wP$ as a test statistic for testing uniformity when the data are progressively type-II censored. Also note that, in Table \ref{TN1} for U(0,1) all the average values of $\xi_{\alpha}^wP$ lie within the interval (0,0.125). 
\begin{table}[h!]
	\centering
	\caption{Critical points of $\xi_{\alpha}^wP$.}\label{TN2}
	\begin{tabular}{p{1cm} p{1cm} p{3cm} p{2cm} p{2cm}}
		\hline
		$n$ & $r$ & Schemes & $D_{0.025:r:n}$ & $D_{0.975:r:n}$\\
		\hline
		20 & 10 & ($10,0*9$)   & 0.04620 & 0.10704\\
		&    & ($0*9,10$)   & 0.00712 & 0.04500\\
		&    & ($5,0*8,5$)  & 0.01471 & 0.07560\\
		&    & ($1*10$)     & 0.01991 & 0.10209\\[1ex]
		& 15 & ($5,0*14$)   & 0.05614 & 0.10440\\
		&    & ($0*14,5$)   & 0.02711 & 0.08558\\
		&    &($0*7,5,0*7$) & 0.05259 & 0.10666\\
		&    & ($3,0*13,2$) & 0.04077 & 0.09848\\[1ex]
		30 & 10 & ($20,0*9$)   & 0.04466 & 0.10752\\
		&    & ($0*9,20$)   & 0.00210 & 0.10866\\
		&    &($10,0*8,10$) & 0.00632 & 0.04544\\
		&    & ($2*10$)     & 0.00975 & 0.08349\\
		& 20 & ($10,0*19$)  & 0.06153 & 0.10184\\
		&    & ($0*19,10$)  & 0.02377 & 0.07049\\
		&    & ($5,0*18,5$) & 0.03758 & 0.08940\\[1ex]
		50 & 30 & ($20,0*29$)  & 0.06725 & 0.09926\\
		&    & ($0*29,20$)  & 0.02200 & 0.05540\\
		&    &($10,0*28,10$)& 0.03750 & 0.08060\\\hline
	\end{tabular}
\end{table}
We reject the hypothesis if $\xi_{\alpha}^wP$ is large or small. The critical region for significance level $\gamma$ is $$\xi_{\alpha}^wP\leq D_{1-\frac{\gamma}{2}:r:n}\;\;\;\text{or}\;\;\;\xi_{\alpha}^wP\geq D_{\frac{\gamma}{2},r:n},$$ where $D_{\gamma}$ is the $\gamma$-th quantile point of the empirical distribution of $\xi_{\alpha}^wP$. We simulate 5000 PC-II censored samples from U(0,1) distribution and calculate the critical points at 5\% level of significance. The critical points are provided in Table \ref{TN2}. We compute power of the test for various progressive censoring schemes and for $\alpha$ = 2. We consider the same alternatives given by Stephens (1974). We generate 5000 PC-II censored samples from these alternatives and compute the power of the test and present them in Table \ref{TN3}. From the table we observe that the proposed test attains the nominal level of significance for all the censoring schemes. As the sample size increases the power also increases. In general, the performance of the proposed test is good for alternatives $A_j$'s and $C_j$'s. We have observed this phenomenon for complete data as well. Proposed test also performs well for alternative $B_2$. For alternatives $B_{1.5}$ and $B_2$, the test works well when sample size is high.

\begin{table}[h!]
	\centering
	\caption{Power of the test for various schemes for $\alpha$ = 2.}\label{TN3}
	\begin{tabular}{|c c c |c| c c| c c c| c c|}
		\hline
		& & & \multicolumn{8}{c|}{Alternatives}\\
		\hline
		$n$ & $r$ & Schemes & U(0,1) & $A_{1.5}$ & $A_2$ & $B_{1.5}$ & $B_2$ & $B_2$ & $C_{1.5}$ & $C_2$\\
		\hline
		20 & 10 & ($10,0*9$) & 0.054 & 0.154 & 0.373 & 0.097 & 0.214 & 0.556 & 0.109 & 0.201\\
		&    & ($0*9,10$)   & 0.053 & 0.185 & 0.488 & 0.013 & 0.006 & 0.010 & 0.263 & 0.496\\
		&    & ($5,0*8,5$)  & 0.049 & 0.168 & 0.433 & 0.023 & 0.030 & 0.096 & 0.175 & 0.370\\
		&    & ($1*10$)     & 0.050 & 0.135 & 0.299 & 0.034 & 0.057 & 0.154 & 0.106 & 0.184\\[1ex]
		& 15 & ($5,0*14$)   & 0.048 & 0.207 & 0.553 & 0.137 & 0.352 & 0.818 & 0.130 & 0.304\\
		&    & ($0*14,5$)   & 0.051 & 0.223 & 0.578 & 0.057 & 0.123 & 0.494 & 0.171 & 0.374\\
		&    &($0*7,5,0*7$) & 0.047 & 0.183 & 0.465 & 0.107 & 0.296 & 0.709 & 0.114 & 0.241\\
		&    & ($3,0*13,2$) & 0.050 & 0.214 & 0.570 & 0.097 & 0.257 & 0.710 & 0.152 & 0.314\\[1ex]
		30 & 10 & ($20,0*9$)   & 0.049 & 0.148 & 0.348 & 0.088 & 0.200 & 0.496 & 0.097 & 0.198\\
		&    & ($0*9,20$)   & 0.051 & 0.162 & 0.434 & 0.015 & 0.002 & 0.001 & 0.146 & 0.319\\
		&    &($10,0*8,10$) & 0.048 & 0.167 & 0.437 & 0.008 & 0.004 & 0.004 & 0.231 & 0.464\\
		&    & ($2*10$)     & 0.0.047 & 0.142 & 0.317 & 0.013 & 0.014 & 0.025 & 0.152 & 0.270\\
		& 20 & ($10,0*19$)  & 0.052 & 0.290 & 0.714 & 0.180 & 0.519 & 0.947 & 0.183 & 0.424\\
		&    & ($0*19,10$)  & 0.050 & 0.322 & 0.762 & 0.046 & 0.124 & 0.529 & 0.255 & 0.544\\
		&    & ($5,0*18,5$) & 0.053 & 0.324 & 0.750 & 0.110 & 0.335 & 0.832 & 0.185 & 0.418\\[1ex]
		50 & 30 & ($20,0*29$)  & 0.053 & 0.461 & 0.916 & 0.318 & 0.787 & 0.999 & 0.264 & 0.585\\
		&    & ($0*29,20$)  & 0.054 & 0.492 & 0.939 & 0.044 & 0.144 & 0.650 & 0.393 & 0.716\\
		&    &($10,0*28,10$)& 0.049 & 0.470 & 0.920 & 0.150 & 0.481 & 0.956 & 0.256 & 0.596\\\hline
		
	\end{tabular}
\end{table}

\subsection{Data Analysis} The problem of testing uniformity can be implemented in any general foodness-of-fit test. For example, consider the problem of testing 

\begin{eqnarray}
H_0:X\sim f(X)\;\;\;\;vs.\;\;\;\;H_1:X\not\sim f(X).\nonumber
\end{eqnarray}
Using the transformation $Y=F(X)$, the above problem becomes 

\begin{eqnarray}
H_0:Y\sim U(0,1)\;\;\;\;vs.\;\;\;\;H_1:Y\not\sim U(0,1).\nonumber
\end{eqnarray}

\noindent Now consider the failure times for insulating fluid dataset given in Subsection \ref{D1}. This dataset follows exponential distribution with MLE, $\hat{\lambda}$ = 0.0696. Now, if we want to test exponentiality based on this dataset, then using the transformation $Y=1-\exp(-\hat{\lambda}X)$, the problem reduces to testing uniformity. The transformed data are as follows:

\begin{table}[h!]
	\centering
	\begin{tabular}{cccccccc}
		0.01314, & 0.05284, & 0.06463, & 0.08714, & 0.17592, & 0.19743, & 0.25087, & 0.27750,\\ 0.28650, & 0.36390, & 0.40044, & 0.42736, & 0.43763, & 0.56802, & 0.89028, & 0.89600,\\ 0.90559, & 0.92231 & 0.99374.
	\end{tabular}
\end{table}

\noindent Based on these 19 observations, the estimated WCRE value is 0.1534. Note that the true value of WCRE for U(0,1) distribution is 0.1389. The 5\% critical points for WCRE-based uniformity test with $n$ = 19 are (0.09423, 0.15482). Since, the estimated test statistic lies within the critical points, we cannot reject the null hypothesis that $Y$ is uniformly distributed over [0,1]. This implies that we cannot reject the null hypothesis that $X$ is exponentially distributed. Similarly, using WCRTE we can perform uniformity test. For $\alpha$ = 2, the value of the estimated test statistic is 0.0880. The true value of WCRTE for U(0,1) is 0.0834. The 5\% critical points are (0.05453, 0.09871). Since the estimated value of WCRTE lies between the critical points, we cannot reject the null hypothesis of uniformity of $Y$ over [0,1] and consequently exponentiality of $X$.\\

We now apply the proposed uniformity test for PC-II censored samples. We generate PC-II censored data from the above dataset with $r$ = 10. The generated observations and the associated censoring schemes are provided below.

\begin{itemize}
	\item Scheme 1: ($R_1=9, R_2=\cdots=R_{10}=0$).\\
	\{0.0134, 0.05284, 0.06463, 0.19743, 0.25087, 0.28650, 0.40044, 0.42736, 0.56802, 0.89028\}
	
	\item Scheme 2: ($R_1=\cdots=R_{9}=0,R_{10}=9$).\\
	\{0.01314, 0.05284, 0.06463, 0.08714, 0.17592, 0.19743, 0.25087, 0.27750, 0.28650, 0.36390\}
	
	\item Scheme 3: ($R_1=5, R_2=\cdots=R_9=0, R_{10}=4$).\\
	\{0.01314, 0.05284, 0.06463, 0.08714, 0.17592, 0.25087, 0.27750, 0.40044, 0.42736, 0.43763\}
	
	\item Scheme 4: ($R_1=\cdots=R_9=1, R_{10}=0$).\\
	\{0.01314, 0.05284, 0.06463, 0.08714, 0.17592, 0.19743, 0.25087, 0.36390, 0.40044, 0.89600\}
\end{itemize}

\noindent We want to test the hypothesis that the PC-II samples come from U(0,1) distribution. We compute the test statistic and 5\% critical points for every schemes and the obtained results are provided in Table \ref{TN4}. From the table we observe that the proposed test accepts the null hypothesis for all the schemes which is accurate.

\begin{table}[h!]
	\centering
	\caption{Uniformity test resluts for PC-II censoring.}\label{TN4}
	\begin{tabular}{p{2cm} p{3cm} p{4cm} p{2cm}}
		\hline
		Schemes & Test Statistic & Critical Points & Decision\\[1ex]\hline
		1 & 0.07115 & (0.04477, 0.10707) & Accept\\
		2 & 0.01392 & (0.00810, 0.05040) & Accept\\
		3 & 0.02196 & (0.01771, 0.08343) & Accept\\
		4 & 0.09124 & (0.02833, 0.11140) & Accept\\\hline
		
	\end{tabular}
\end{table}

\section{Conclusion}\label{S7} In this paper, we study some new properties of WCRTE measure and develop four non-parametric estimators of WCRTE. We investigate asymptotic properties of the proposed estimators as well. We propose estimators for WCRE measure and study their asymptotic properties. Also, we suggested an estimator for WCRTE under progressively type-II censored samples. We compare the performances of these estimators using simulation and develop two new uniformity tests for complete samples. Power of these tests are compared with some popular existing tests and it is observed that proposed tests perform better compared to the other tests in majority of the cases. Also, we proposed an uniformity test under progressive type-II censoring. We investigated its performance by evaluating the power and we have observed that the uniformity test for progressive type-II censoring works quite well.\\

We use one estimator of WCRTE and WCRE to construct the uniformity test. Other estimators may also be used for this purpose. It will be interesting to see, for which value of the window size of these estimators gives the highest power. We evaluate the power of the uniformity test under progressive type-II censoring when $\alpha$ = 2. Computation became very intensive when it came to progressive type-II censoring. Note that, for fixed values of $n$ and $r$, there exists $n-1\choose r-1$ number of different censoring schemes. Therefore the number of progressive censoring schemes becomes very large even for moderate $n$ and $r$. It will be interesting to study the power of the test for different choices of $\alpha$ and for other censoring schemes. More work is needed in this direction.

\section*{Conflicts of Interest} The authors declare no conflict of interest.

\section*{Funding} No funding is received for this work.

\end{document}